\documentclass[
a4paper,
11pt, 
DIV=12,
toc=bibliography, 
headsepline, 
parskip=half-,
abstract=true,
]{scrartcl}

\usepackage[english]{babel}
\usepackage[utf8]{inputenc}
\usepackage[T1]{fontenc}
\usepackage[english,cleanlook]{isodate}
\usepackage{lmodern}
\usepackage[babel=true]{microtype}
\usepackage{amsmath}
\usepackage{mathtools}
\usepackage{amssymb}
\usepackage{amsthm}
\usepackage{mathrsfs}
\usepackage{scrlayer-scrpage} 
\usepackage{enumitem}
\usepackage{tabularx}
\usepackage{tikz} 
\usepackage{tikz-3dplot} 
\usetikzlibrary{positioning,plotmarks}   

\ihead{G.\,Franz, D.\,Ketover, M.\,B.\,Schulz}
\ohead{\csname @title\endcsname}
\setkomafont{pageheadfoot}{\footnotesize}
\providecommand{\R}{\mathbb{R}}
\providecommand{\Z}{\mathbb{Z}}
\providecommand{\N}{\mathbb{N}}
\providecommand{\B}{\mathbb{B}}
\providecommand{\K}{\mathbb{K}}
\providecommand{\Sp}{\mathbb{S}}
\providecommand{\dih}{\mathbb{D}}
\providecommand{\apr}{\mathbb{A}}
\providecommand{\pri}{\mathbb{P}}
\providecommand{\pyr}{\mathbb{Y}}
\providecommand{\st}{\, :\ }
\providecommand{\refl}{\textsf{R}} 
\providecommand{\hsd}{\mathscr{H}}
\providecommand{\gsum}{\mathfrak{g}}  
\providecommand{\bsum}{\mathfrak{b}} 
\providecommand{\compgamma}{K}
\providecommand{\FKS}{\Theta}
\providecommand{\KKMS}{\Sigma^{\mathrm{KKMS}}} 
\newcommand{\vard}{\ensuremath{\boldsymbol{\mathrm{F}}}}
\newcommand{\gelem}{\ensuremath{h}}

\DeclarePairedDelimiter\abs{\lvert}{\rvert}
\DeclarePairedDelimiter\nm{\lVert}{\rVert}
\DeclarePairedDelimiter\sk{\langle}{\rangle}
\DeclarePairedDelimiter\interval{]}{[}
\DeclarePairedDelimiter\Interval{[}{[}
\DeclarePairedDelimiter\intervaL{]}{]}

\DeclareMathOperator{\Ogroup}{O}

\DeclareMathOperator{\ind}{ind}
\DeclareMathOperator{\II}{I\hspace*{-.5pt}I}
\DeclareMathOperator{\Ric}{Ric}

\pgfmathsetmacro{\picscale}{0.39}
\pgfmathsetmacro{\unitscale}{\picscale*\textwidth/2cm}
\providecommand{\FBMS}[2][\picscale]{\draw(0,0,0)node[inner sep=0]{\includegraphics[width=#1\textwidth]{#2}};}
\pgfmathsetmacro{\thetaO}{72}
\pgfmathsetmacro{\phiO}{90-45/2}
\newenvironment{unitball}[1][\picscale]{%
\tdplotsetmaincoords{\thetaO}{\phiO}
\begin{tikzpicture}[scale={#1*\textwidth/2cm},tdplot_main_coords,line cap=round,line join=round,semithick,baseline={(0,0,0)}]
\pgfmathsetmacro{\laengeA}{-135-45/2+\phiO}
\pgfmathsetmacro{\laengeB}{\laengeA+45}
\begin{scope}[black!30]
\path[tdplot_screen_coords,inner color=white,outer color=cyan!5!black!20](0,0)circle(1);
\foreach\laengengrad in {\laengeA,\laengeB,...,\phiO}{
\tdplotsetthetaplanecoords{\laengengrad}
\pgfmathsetmacro{\start}{-atan( 1/(tan(\thetaO)*( cos(\laengengrad)*sin(\phiO)-cos(\phiO)*sin(\laengengrad))))}
\tdplotdrawarc[tdplot_rotated_coords]{(0,0,0)}{1}{180+\start}{360+\start}{}{}
}
\foreach\breitengrad in {-60,-30,...,60}{
\pgfmathsetmacro{\breitenkreisradius}{cos(\breitengrad)}
\pgfmathsetmacro{\hoehe}{sin(\breitengrad)}\tdplotdrawarc{(0,0,\hoehe)}{\breitenkreisradius}{0}{360}{}{}
}
\end{scope}
}{\begin{scope}[black]
\foreach\laengengrad in {\laengeA,\laengeB,...,\phiO}{
\tdplotsetthetaplanecoords{\laengengrad}
\pgfmathsetmacro{\start}{-atan( 1/(tan(\thetaO)*( cos(\laengengrad)*sin(\phiO)-cos(\phiO)*sin(\laengengrad))))}
\tdplotdrawarc[tdplot_rotated_coords]{(0,0,0)}{1}{\start}{180+\start}{}{}
} 
\foreach\breitengrad in {-60,-30,...,60}{
\pgfmathsetmacro{\breitenkreisradius}{1*cos(\breitengrad)}
\pgfmathsetmacro{\hoehe}{1*sin(\breitengrad)}
\pgfmathsetmacro{\visiblerange}{acos(-sign(\breitengrad)*min(1,abs(tan(\breitengrad)*tan(90-\thetaO)))  )}
\ifdim\visiblerange pt>0pt
\tdplotdrawarc{(0,0,\hoehe)}{\breitenkreisradius}{\phiO-90+\visiblerange}{\phiO-90-\visiblerange}{}{}
\fi
}
\draw[tdplot_screen_coords] (0,0) circle (1);
\end{scope}%
\end{tikzpicture}%
}
\makeatletter
\tikzset{
    scale plot marks/.is choice,
    scale plot marks/true/.style={},	
    scale plot marks/false/.code={
        \def\pgfuseplotmark##1{\pgftransformresetnontranslations\csname pgf@plot@mark@##1\endcsname}
    },
every mark/.append style={scale plot marks=false},
plus/.style={mark=+,mark size=2.25pt},
vdash/.style={mark=|,mark size=2.25pt},
hdash/.style={mark=-,mark size=2.25pt},
bullet/.style={mark=*,mark size=1.125pt},
}
\makeatother

\usepackage[initials, alphabetic]{amsrefs}
 
\usepackage[
pdftitle={Genus one critical catenoid},
pdfauthor={Giada Franz, Daniel Ketover and Mario B. Schulz}, 
pdfborder={0 0 0},
]{hyperref}

\theoremstyle{plain}
\newtheorem{theorem}{Theorem}[section]
\newtheorem{lemma}[theorem]{Lemma}
 
\newtheorem{proposition}[theorem]{Proposition}

\theoremstyle{definition}
\newtheorem{definition}[theorem]{Definition}
 
\theoremstyle{remark}
\newtheorem{remark}[theorem]{Remark}

\title{Genus one critical catenoid}
\author{Giada Franz, Daniel Ketover and Mario B. Schulz}
\date{\vspace*{-3ex}}
 
\newcommand\printaddress{{
\setlength{\parindent}{17pt}
\small
\smallskip
\hfill September 2024
\par
{\scshape Giada Franz}
\newline
MIT, Department of Mathematics, 77 Massachusetts Avenue, Cambridge, MA 02139, USA
\newline
\textit{E-mail address:} 
\texttt{gfranz@mit.edu}
\par\medskip
{\scshape Daniel Ketover}
\newline 
Rutgers University, Busch Campus -- Hill Center, 110 Freylinghausen Road, Piscataway, NJ 08854, USA
\newline
\textit{E-mail address:} 
\texttt{dk927@math.rutgers.edu}
\par\medskip
{\scshape Mario B. Schulz}
\newline 
Università di Trento, 
Dipartimento di Matematica, 
via Sommarive 14, 
38123 Povo di Trento, 
Italy 
\newline
\textit{E-mail address:} 
\texttt{mario.schulz@unitn.it}
\par
}} 

\begin{document}

\maketitle

\begin{abstract}
We use variational methods to construct a free boundary minimal surface in the three-dimensional unit ball with genus one, two boundary components and prismatic symmetry. 
Key ingredients are an extension of the equivariant min-max theory to include orientation-reversing isometries and the discovery of a nontrivial two-parameter sweepout.
\end{abstract}

\section{Introduction}

The study of minimal surfaces in a given three-dimensional ambient manifold has long been a central theme in differential geometry. 
In the case where the ambient manifold $M$ has nonempty boundary $\partial M$, it is natural to study minimal surfaces with free boundary, i.\,e.~ciritical points for the area functional among all surfaces $\Sigma\subset M$ with boundary $\partial\Sigma$ constraint to $\partial M$. 
Equivalently, a free boundary minimal surface has vanishing mean curvature while meeting $\partial M$ orthogonally.  
In convex ambient manifolds with nonnegative Ricci curvature, free boundary minimal surfaces are necessarily unstable, which complicates their construction. 
The three-dimensional Euclidean unit ball $\B^3$ is a particularly interesting ambient manifold of this type, because this setting is closely related to the optimization problem for the first Steklov eigenvalue on surfaces with boundary \cite{FraserSchoen2011,FraserSchoen2016,KarpukhinKokarevPolterovich2014,GirouardLagace2021,KarpukhinStern2024}. 
Recently, Karpukhin, Kusner, McGrath and Stern \cite[Theorem~1.2]{KarpukhinKusnerMcGrathStern2024} developed equivariant Steklov eigenvalue optimization methods to prove that 
any compact, orientable surface with boundary can be realised as an embedded free boundary minimal surface with area strictly below $2\pi$ in $\B^3$. 
Other methodologies used to establish existence results include gluing methods and min-max theory, with the latter being the main focus of this article.

Gluing methods have been employed to construct solutions with large topological complexity \cite{FolhaPacardZolotareva2017,KapouleasLi2017,KapouleasMcGrath2023,KapouleasWiygul2017,KapouleasZou2021}, including 
pairs of solutions with the same topology and symmetry group \cite{CarlottoSchulzWiygul2022}, and families of solutions with unbounded area \cite{CSWstackings}. 

Min-max methods on the other hand allow the construction of free boundary minimal surfaces with low topological complexity in various ambient manifolds, not limited to Euclidean balls \cite{GruterJost1986,Ketover2016FBMS,Li2015,CarlottoFranzSchulz2022,FranzSchulz2023,HaslhoferKetover,SchulzEllipsoids}. 
Min-max theory has been pioneered by Almgren--Pitts \cite{Almgren1965, Pitts1981}, Marques--Neves \cite{MarquesNeves2014, MarquesNeves2017},  Simon--Smith \cite{Smith1982} and Colding--De Lellis \cite{ColdingDeLellis2003}, and was adapted to the free boundary setting by Li \cite{Li2015}. 
The equivariant version of min-max theory has been developed in \cite{Ketover2016Equivariant,Ketover2016FBMS,FranzSchulz2023} for the actions of finite groups of orientation-preserving ambient isometries.   
In Theorem~\ref{thm:EquivMinMax}, we extend the theory allowing for arbitrary finite group actions, including those generated by (orientation-reversing) reflections.
 
Combining this extension with a two-parameter min-max scheme we obtain the existence of a ``genus one critical catenoid'' $\FKS\subset\B^3$. 
The surface $\Theta$ is preserved by the prismatic group $\pri_2$ of order $8$ (generated by the reflective symmetries with respect to the three coordinate planes $\{x_1=0\}$, $\{x_2=0\}$, $\{x_3=0\}$, cf.~Remark~\ref{rem:prismatic}). 
The following is our main result:

\begin{figure}%
\begin{unitball}
\draw[dotted]
(0,0,0)--(0,0,0.209594853515861)
(0,0.269050102990999,0)--(0,1,0)
;
\FBMS{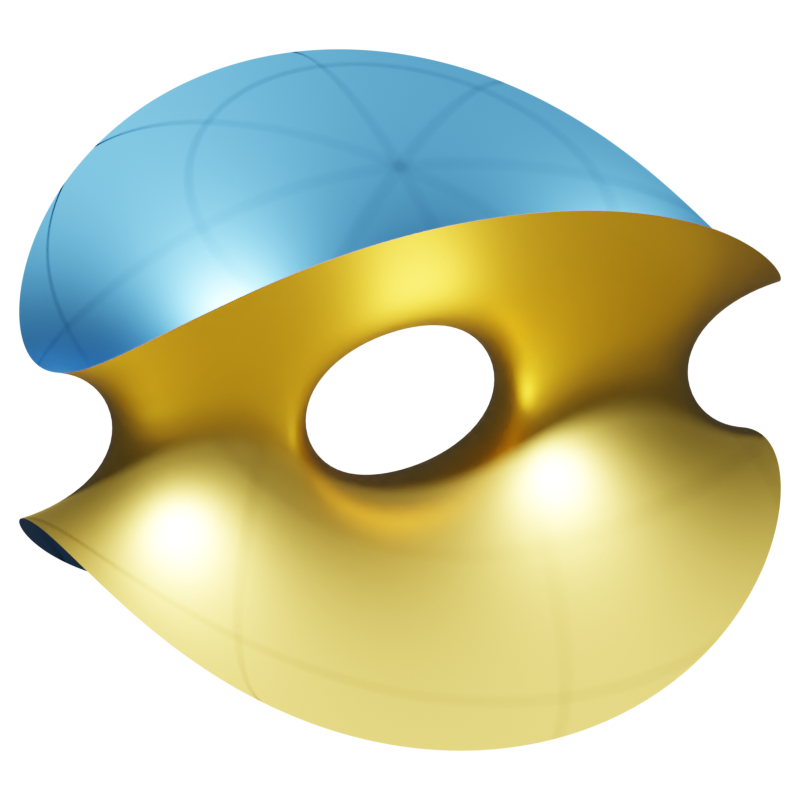}
\draw[dotted]
(0,0,0)--(1,0,0)
(0,0,0)--(0,0.269050102990999,0)
(0,0,0.44)--(0,0,1);
\draw[->](1,0,0)--++(0.3,0,0)node[below right={1ex and 0},inner sep=0]{$x_3$};
\draw[->](0,1,0)--++(0,0.3,0)node[below left={1ex and 0},inner sep=0]{$x_1$};
\draw[->](0,0,1)--++(0,0,0.3)node[below right,inner sep=0]{~$x_2$};
\begin{scope}[black!40]
\clip[tdplot_screen_coords](0,0)circle(1);
\draw[->](0,1,0)--++(0,0.3,0);
\end{scope}
\end{unitball}
\hfill
\begin{unitball}
\FBMS{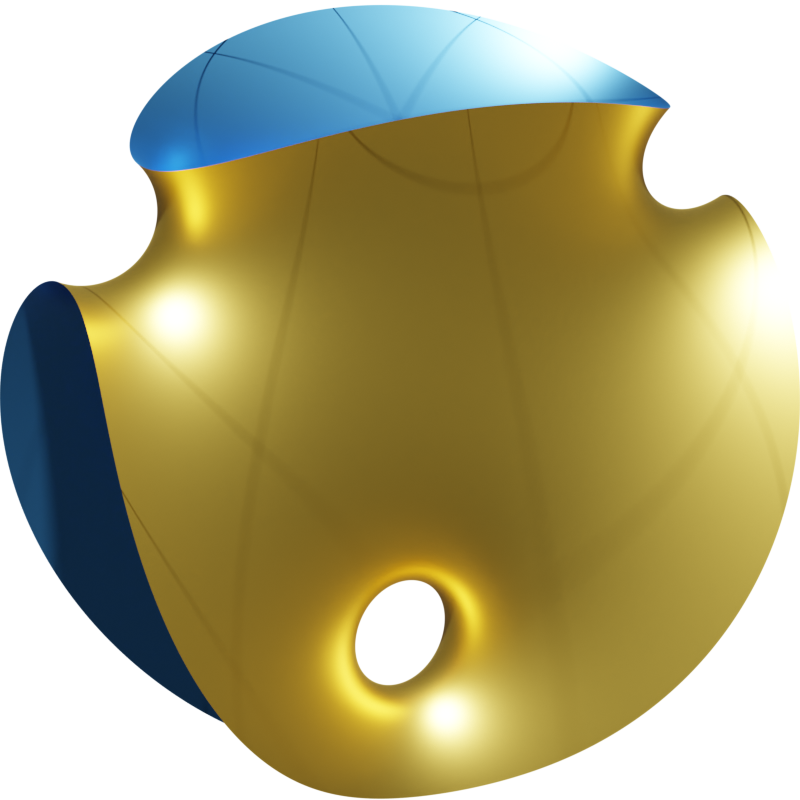}
\end{unitball}
\caption{Left image: Simulation of a $\pri_2$-equivariant genus one catenoid in the unit ball $\B^3$. 
Right image: Less symmetric solution with the same topology but with more area.}%
\label{fig:g1b2}%
\end{figure}
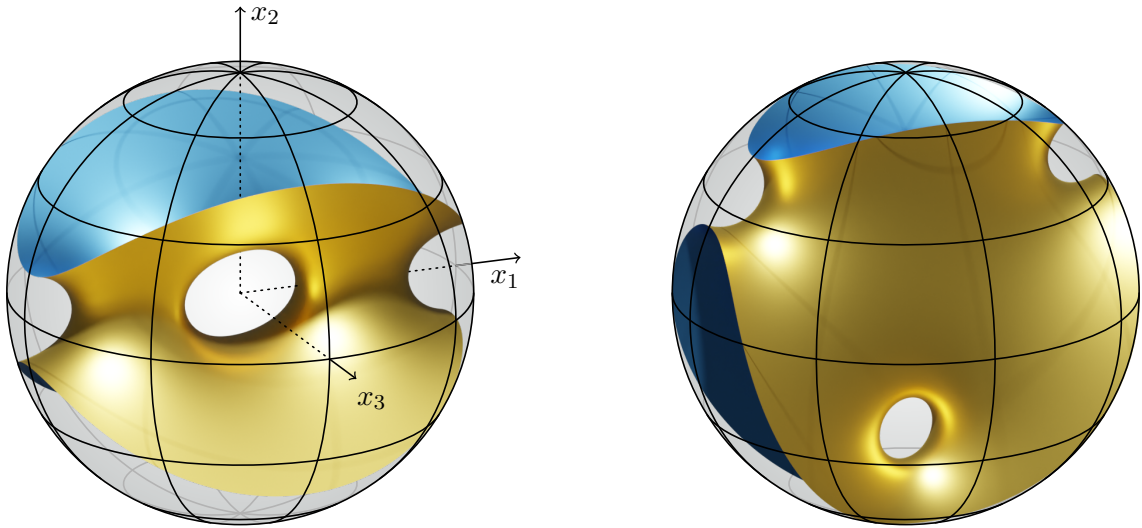

\begin{theorem}\label{thm:ExistenceGenusOneCatenoid}
The Euclidean unit ball $\B^3$ contains an embedded, $\pri_2$-equivariant free boundary minimal surface $\FKS$ with the following properties.
\begin{enumerate}[label={\normalfont(\roman*)},nosep]
\item $\FKS$ has genus one and two boundary components. 
\item The $\pri_2$-equivariant index of $\FKS$ is equal to $2$. 
\item The area of $\FKS$ is strictly greater than that of the critical catenoid and strictly less than $2\pi$.
\item The first nonzero Steklov eigenvalue on $\FKS$ is equal to $1$.
\end{enumerate}
\end{theorem}
 
The $\pri_2$-equivariant genus one catenoid $\FKS$ constructed in Theorem~\ref{thm:ExistenceGenusOneCatenoid} should be compared with the $\pyr_2\cong\Z_2\times\Z_2$-equivariant free boundary minimal surface $\KKMS$ with genus one and two boundary components found in \cite[Theorem~1.2]{KarpukhinKusnerMcGrathStern2024}. 
The numerical simulations depicted in Figure~\ref{fig:g1b2} indicate that $\KKMS$ has smaller symmetry group but greater area than that of $\FKS$ (cf.~Table~\ref{table:complexity}). 
In particular, we expect the two solutions $\FKS$ and $\KKMS$ to be distinct.  
We also expect this pair to be of the simplest topological type for which the phenomenon of topological nonuniqueness of free boundary minimal surfaces in $\B^3$ occurs. 

It remains an open question whether the critical catenoid is the unique free boundary minimal surface in $\B^3$ with the second least area and second least Morse index after the equatorial disc. 
Ordering the known embedded free boundary minimal surfaces in $\B^3$ by geometric complexity, it is likely that the surface $M_1$ with genus one and connected boundary constructed in \cite{CarlottoFranzSchulz2022} is the next surface on the list after the critical catenoid (see Table \ref{table:complexity}). 
We expect that the genus one catenoid $\FKS$ obtained in Theorem \ref{thm:ExistenceGenusOneCatenoid} is the fourth simplest example in $\B^3$, followed by the free boundary minimal trinoid constructed in \cite{FranzSchulz2023} 
and the aforementioned solution $\KKMS$ found in \cite{KarpukhinKusnerMcGrathStern2024}.  
Moreover, we conjecture that $\FKS$ has Morse index equal to $6$. 
The critical catenoid is known to have Morse index equal to $4$ \cite{Devyver2019,Tran2020,SmithZhou2019}. 
Existence of a solution with Morse index $5$ is also known \cite[§\,2]{FranzSchulz2023}, and this surface conjecturally coincides with $M_1$.  

\begin{table} 
\caption{Free boundary minimal surfaces in $\B^3$ by order of complexity.
Data marked with ${}^{*}$ are conjectural based on numerical simulations. 
See Table~\ref{table:group} for a list of symmetry groups.  
} 
\label{table:complexity}
\smallskip
\begin{tabular}{
>{\raggedright\arraybackslash}p{0.345\textwidth-2\tabcolsep} 
>{\centering\arraybackslash}p{0.13\textwidth-3\tabcolsep}  
>{\centering\arraybackslash}p{0.13\textwidth-1\tabcolsep}  
>{\centering\arraybackslash}p{0.13\textwidth-2\tabcolsep}  
>{\centering\arraybackslash}p{0.13\textwidth-2\tabcolsep}  
>{\centering\arraybackslash}p{0.13\textwidth-2\tabcolsep}
}
free boundary minimal surface             & genus & boundaries & symmetry      &  area           & index   \\\hline\hline
equatorial disc                           & $0$   & $1$        & $\Ogroup(2)$  &  $\pi$          & $1$     \\\hline   
critical catenoid $\K$                    & $0$   & $2$        & $\Ogroup(2)$  &  $1.6671\,\pi$  & $4$     \\\hline
$M_1$ from \cite{CarlottoFranzSchulz2022} & $1$   & $1$        & $\apr_2$      &  $1.8032\,\pi^{\mathrlap*}$ &  $5^{\mathrlap*}$    \\\hline
genus one critical catenoid $\FKS$                 & $1$   & $2$        & $\pri_2$      &  $1.8559\,\pi^{\mathrlap*}$ &  $6^{\mathrlap*}$    \\\hline
trinoid from \cite{FranzSchulz2023}       & $0$   & $3$        & $\pri_3$      &  $1.9117\,\pi^{\mathrlap*}$ &  $6^{\mathrlap*}$    \\\hline
$\KKMS$ from \cite{KarpukhinKusnerMcGrathStern2024} & $1$   & $2$        & $\pyr_2$      &  $1.9460\,\pi^{\mathrlap*}$ &  $7^{\mathrlap*}$    \\\hline
\end{tabular}
\end{table}

Let us sketch some of the main ideas involved in the proof of Theorem~\ref{thm:ExistenceGenusOneCatenoid}. 
Given Cartesian coordinates $x_1,x_2,x_3$ in $\R^3\supset\B^3$, there are three different realizations of the critical catenoid in $\B^3$ which are invariant under reflection across the three coordinate planes. 
We denote the one which is disjoint from the $x_i$-axis by $\K_{x_i}$ for each $i\in\{1,2,3\}$. 
The first key insight is that there is a one-parameter $\pri_2$-equivariant family $\{\Sigma_s\}_{s\in [0,1]}$ of genus one surfaces with two boundary components interpolating between $\K_{x_3}$ (together with arcs) and $\K_{x_1}$ (together with arcs). 
Roughly speaking, near $\K_{x_3}$, the surface is obtained by adding two ribbons connecting the two boundary components of $\K_{x_3}$ along the intersection $\partial\B\cap \{x_2=0\}$
(see Figure~\ref{fig:sweepout}). 
Performing a one-parameter equivariant min-max procedure on this family, however, does not produce a genus one catenoid as there is no guarantee that the width of this family is greater than the area of the critical catenoid. 
In fact we show in Proposition~\ref{prop:EquivIndexMoreThan2} that the desired surface must have equivariant index at least $2$ and therefore it would not be possible to obtain it with a one-parameter min-max procedure. 

To remedy this,  we extend the family to a two-parameter family $\{\Sigma_{s,t}\}_{(s,t)\in [0,1]^2}$ by considering for each $s$, a family of surfaces that foliates each component of $\B^3\setminus \Sigma_s$.
In this way,  for each $s$, the family $\{\Sigma_{s, t}\}_{t\in [0,1]}$ gives a nontrivial sweepout of the ball $\B^3$. 
We can then use Lusternick--Schnirelman theory to ensure the following dichotomy: either the width of this two-parameter family is indeed greater than the area of the critical catenoid or else there is a one parameter family of genus one catenoids with area less than that of the critical catenoid. 
In both cases we obtain a new genus one free boundary catenoid. 
The fact that the equivariant index is at least $2$ then rules out the second case in the dichotomoy. 
Key ingredients in controlling the topological type of our limiting free boundary minimal surface are  
the topological lower semicontinuity result from \cite{FranzSchulz2023} and the uniqueness of the critical catenoid among embedded free boundary minimal annuli in $\B^3$ with prismatic symmetry \cite{McGrath2018}*{Theorem~1}. 
The fact that the prismatic symmetry group includes orientation-reversing isometries motivates our extension  
of the equivariant min-max procedure to accommodate all finite group of isometries without requiring them to be orientation-preserving (see Theorem~\ref{thm:EquivMinMax}).  

Two-parameter families that ``flip'' a Heegaard splitting (or optimal foliation) to one with the opposite orientation were introduced in \cite{Ketover2022} to produce index $2$ minimal surfaces in three-manifolds in generality.  Here the situation is novel in that we consider two-parameter families that \emph{rotate} a critical catenoid by angle $\pi/2$ in addition to interchanging the sides.  As a technical point, by allowing the initial and final foliations to slide, we may interpolate instead between differently oriented annuli instead of the critical catenoid which gives more flexibility in constructing the two-parameter family.  

\begin{remark}
Our generalization of the equivariant min-max procedure also applies in the closed case (e.\,g.~in the unit sphere $\Sp^3$). 
Previous equivariant min-max constructions of minimal surfaces use orientation-preserving isometry groups, such as the dihedral group (see Table~\ref{table:group}). 
In many instances, these constructions can now be replicated verbatim with a larger symmetry group, such as the prismatic or antiprismatic group, which includes orientation-reversing isometries and is typically expected to be the full symmetry group of the surfaces in question (see e.\,g. \cite{FranzSchulz2023}*{Remark~5.2}). 
This is for example the case for the surfaces constructed in \cite{Ketover2016Equivariant}*{Section~6.2} (see Remark~6.1 therein), \cite{Ketover2016FBMS}*{Theorem~1.1}, \cite{BuzanoNguyenSchulz2021}*{Theorem~1.2}, \cite{CarlottoFranzSchulz2022}*{Theorem~1.1}, \cite{FranzSchulz2023}*{Theorem~5.1} and \cite{SchulzEllipsoids}*{Theorems~1.1--1.3}.
\end{remark}

\begin{remark}
In the case when the finite group of isometries is generated by reflection symmetries, the equivariant min-max approach is related to a (nonequivariant) min-max procedure in a fundamental domain of the group action -- a locally wedge-shaped manifold. 
Almgren--Pitts min-max theory in locally wedge-shaped manifolds has recently been studied in \cite{MazurowskiWang2023}. 
However, we consider here the Simon--Smith setting rather than the Almgren--Pitts one. 
We also refer to \cite{Wang2023} for equivariant Almgren--Pitts min-max theory in the free boundary setting.
\end{remark}

The paper is organized as follows:
\begin{itemize}[nosep]
\item In Section~\ref{sec:minmax}, we state and prove our extension of the equivariant min-max theorem.
\item In Section~\ref{sec:prismaticsurfaces}, we investigate the general structure of surfaces with prismatic symmetry.
\item In Section~\ref{sec:sweepout}, we design the two-parameter sweepout intended for the min-max procedure.
\item In Section~\ref{sec:topology}, we apply the equivariant min-max theorem with the sweepout constructed in the previous section, and obtain a free boundary minimal surface with the desired properties, concluding the proof of Theorem~\ref{thm:ExistenceGenusOneCatenoid}. 
\end{itemize}

\paragraph{Acknowledgements.} 
G.\,F.~was partially supported by NSF grant DMS-2405361. Moreover, part of this work was performed while G.\,F. was in residence at the Simons Laufer Mathematical Sciences Institute (formerly MSRI) during the Fall 2024 semester, supported by NSF grant DMS-1928930.
D.\,K.~was partially supported by NSF DMS-1906385. 
M.\,S.~has received funding from the European Research Council (ERC) under the European Union’s Horizon 2020 research and innovation programme (grant agreement No. 947923).

\section{Equivariant min-max theory with reflection symmetries} \label{sec:minmax}

\subsection{Symmetry groups}

Let $M$ be a $3$-dimensional Riemannian manifold and let $G$ be a finite group of isometries of $M$. 
The \emph{isotropy group} at $x\in M$ is defined as 
\begin{align*}
G_x&=\{h\in G\st h(x)=x\}.
\intertext{Equivalently, we say that \emph{$x$ is of isotropic type $G_x$}. 
The \emph{singular locus} of $G$ is then defined as}
\mathcal{S}&=\{x\in M \st G_x\neq\{\mathrm{id}\}\}.
\end{align*} 
Given any $x\in M$ the isotropy group $G_x$ acts on the tangent space $T_x M$ as a finite subgroup of the orthogonal group $\Ogroup(3)$. 
The finite subgroups of $\Ogroup(3)$ are fully characterized (see e.\,g.~\cite{ConwaySmith2003}*{§\,3}). 
We list them in Table~\ref{table:group} and include a description of their singular loci. 
In each case, the isotropic type of the origin is equal to the entire subgroup; hence, we mention only the isotropic types of points other than the origin. 
Note that some of the groups (e.\,g.~the pro-antiprismatic group and the full symmetry groups of the platonic solids) contain glide reflections (also called rotoreflections), which have only one fixed point at the origin.  

The literature uses various notation systems to refer to finite subgroups of $\Ogroup(3)$, with no universally accepted standard. 
We employ the notation introduced in \cite[§\,2]{CarlottoSchulzWiygul2022} as well as Conway--Thurston's orbifold notation, which encodes the transformation types in the group action and the isotropic types of points in the singular locus. 
More precisely, the asterix $\mathord*$ indicates a reflection, while a number $n$ represents a rotation of angle $2\pi/n$. 
When a number $n$ appears on the right of an asterix, then the singular locus contains points of isotropic type $\mathord*nn$. 
For example, $*11$ denotes the group generated by the reflection with respect to one plane.
The orbifold notation $\mathord*22n$ for the prismatic group $\pri_n$ suggests the presence of reflections, rotations of angles $\pi$ respectively $2\pi/n$, and points of isotropic types $\mathord*22$ and $\mathord*nn$.

\begin{table} 
\caption{List of finite subgroups of $\Ogroup(3)$. The prefix ``chiro-'' stands for \emph{orientation-preserving}, while the prefix ``holo-'' stands for \emph{whole}. 
}
\label{table:group}
\smallskip
\begin{tabular}{
>{\raggedright\arraybackslash}p{0.3\textwidth-2\tabcolsep} 
>{\centering\arraybackslash}p{0.1\textwidth-2\tabcolsep} 
>{\raggedleft\arraybackslash}p{0.1\textwidth-2\tabcolsep} 
>{\centering\arraybackslash}p{0.198\textwidth-2\tabcolsep} 
>{\raggedright\arraybackslash}p{0.3\textwidth-2\tabcolsep}}
name          & \multicolumn{2}{c}{notation(s)} & order& singular locus \\\hline\hline
cyclic               &$nn$           & $\Z_n$   &  $n$ & $1$ line  of type $nn$ \\\hline
dihedral             &$22n$          & $\dih_n$ & $2n$ & $n$ lines of type $22$,\newline 
                                                         $1$ line of type $nn$  \\\hline
(holo-)pyramidal     &$\mathord*nn$  & $\pyr_n$ & $2n$ & $1$ line of type $\mathord*nn$,\newline 
                                                         $n$ planes of type $\mathord*11$ \\\hline
(holo-)prismatic     &$\mathord*22n$ & $\pri_n$ & $4n$ & $n$ lines of type $\mathord*22$,\newline  
                                                         $1$ line of type $\mathord*nn$,\newline 
                                                         $n+1$ planes of type $*11$ \\\hline
(holo-)antiprismatic &$2\mathord*n$  & $\apr_n$ & $4n$ & $n$ lines of type $22$,\newline 
                                                         $1$ line of type $*nn$,\newline 
                                                         $n$ planes of type $*11$ \\\hline
pro-prismatic        &$n\mathord*$   &          & $2n$ & $1$ line of type $nn$,\newline 
                                                         $1$ plane of type $\mathord*11$ \\\hline
pro-antiprismatic    &$n\times$      &          & $2n$ & $1$ line of type $nn$ \\\hline
chiro-tetrahedral    &$332$          &          & $12$ & $3$ lines of type $22$,\newline
                                                         $4$ lines of type $33$ \\\hline
chiro-octahedral     &$432$          &          & $24$ & $6$ lines of type $22$,\newline 
                                                         $4$ lines of type $33$,\newline 
                                                         $3$ lines of type $44$ \\\hline
chiro-icosahedral    &$532$          &          & $60$ & $15$ lines of type $22$,\newline 
                                                         $10$ lines of type $33$,\newline 
                                                         $12$ lines of type $55$ \\\hline
pyritohedral         &$3\mathord*2$  &          & $24$ &  $3$ lines of type $\mathord*22$,\newline
                                                         $4$ lines of type $33$ \newline
                                                         $3$ planes of type $\mathord*11$ \\\hline
(holo-)tetrahedral   &$\mathord*332$ &          & $24$ & $3$ lines of type $\mathord*22$,\newline
                                                         $4$ lines of type $\mathord*33$,\newline
                                                         $6$ planes of type $\mathord*11$  \\\hline
(holo-)octahedral    &$\mathord*432$ &          & $48$ & $6$ lines of type $\mathord*22$,\newline 
                                                         $4$ lines of type $\mathord*33$,\newline 
                                                         $3$ lines of type $\mathord*44$,\newline 
                                                         $9$ planes of type $\mathord*11$   \\\hline
(holo-)icosahedral   &$\mathord*532$ &          & $120$& $15$ lines of type $\mathord*22$,\newline 
                                                         $10$ lines of type $\mathord*33$,\newline 
                                                         $12$ lines of type $\mathord*55$,\newline 
                                                         $15$ planes of type $\mathord*11$  \\\hline\hline
\end{tabular}
\end{table}

As a corollary of the full characterization of finite subgroups of $\Ogroup(3)$, we obtain the following description of the singular locus of a finite group of isometries acting on a three-dimensional Riemannian manifold. 

\begin{proposition} \label{prop:SingLocusStrata}
Let $G$ be a finite group of isometries acting on a three-dimensional Riemannian manifold $M$. Then the singular locus $\mathcal{S}$ can be written as $\mathcal{S}=\mathcal{S}_0\cup\mathcal{S}_1\cup\mathcal{S}_2$, where:
\begin{itemize}[nosep]
\item $\mathcal{S}_0$ is a finite set of isolated points;
\item $\mathcal{S}_1$ is a union of (geodesic) lines of isotropic type $nn$ or $\mathord*nn$ for some $n\geq 2$;
\item $\mathcal{S}_2$ is a union of (totally geodesic) surfaces of isotropic type $\mathord*11$.
\end{itemize}
\end{proposition}

\begin{remark}[Prismatic group action]\label{rem:prismatic} 
The group relevant for Theorem~\ref{thm:ExistenceGenusOneCatenoid} is the prismatic group $\pri_2$ of order $8$ acting on $\B^3$. 
It is isomorphic to $\Z_2\times\Z_2\times\Z_2$. 
In general the prismatic group $\pri_n$ is isomorphic to $\Z_2\times\dih_n$ for any $2\leq n\in\N$. 
We typically equip $\B^3$ with Cartesian coordinates $x_1,x_2,x_3$ such that the action of $\pri_n$ on $\B^3$ is generated by the reflection across the plane $\{x_3=0\}$, the rotation of angle $\pi$ around the $x_1$-axis and the rotation of angle $2\pi/n$ around the $x_3$-axis.
\end{remark}

\subsection{Equivariant free boundary minimal surfaces}

We adopt the notation and definitions from \cite{Franz2022}*{Sections~11 and~12}. 
For completeness, we will recall the relevant ones and prove some additional lemmata. 
As before, $M$ is a compact, $3$-dimensional Riemannian manifold with boundary and $G$ a finite group of isometries of $M$.  

\begin{definition}\label{defn:varifolds}
We denote by $\mathcal{V}^2_G(M)$ the set of $2$-dimensional $G$-equivariant varifolds supported in $M$, endowed with the weak topology. Moreover, let us denote by $\vard$ a metric metrizing this topology (see \cite{Pitts1981}*{pp.~66} or \cite{MarquesNeves2014}*{pp.~703}).
We denote by $\nm{V}$ the Radon measure in $M$ associated to $V$, and we call $\nm{V}(M)$ the area of $V$.
\end{definition}

\begin{definition}
Given a $G$-equivariant varifold $V\in\mathcal{V}^2_G(M)$, we say that $V$ is \emph{free boundary stationary} if $\delta V(X)=0$ for all vector fields $X$ on $M$  tangent to $\partial M$, meaning $X(x)\in T_x\partial M$ for all $x\in \partial M$. Here, $\delta V(X)$ is the first variation of the area of $V$ along $X$ as defined e.\,g.~in \cite{Franz2022}*{Definition~1.1.1}.
By Palais' \cite{Palais1979} principle of symmetric criticality, it is equivalent for $V$ to satisfy $\delta V(X) = 0$ for all $G$-equivariant vector fields $X$ on $M$ tangent to $\partial M$ (see also \cite{Ketover2016Equivariant}*{Lemma~3.8} for the result in this setting).
\end{definition}

\begin{definition}
Let $\Sigma\subset M$ be a $G$-equivariant free boundary minimal surface in $M$ 
(which means that the associated varifold is free boundary stationary) 
and let $\Gamma_G(N\Sigma)$ denote the sections of the normal bundle of $\Sigma$ obtained as restriction to $\Sigma$ of $G$-equivariant vector fields in $M$.
Let $Q^\Sigma$ be the quadratic form such that $Q^\Sigma(Y,Y)$ is the second variation of the area of $\Sigma$ along $Y$, for every $Y\in\Gamma_G(N\Sigma)$.
Then, the \emph{$G$-equivariant (Morse) index} $\ind_G(\Sigma)$ of $\Sigma$ is defined as the maximal dimension of a linear subspace of $\Gamma_G(N\Sigma)$ where $Q^\Sigma$ is negative definite.
Moreover, we say that $\Sigma$ is \emph{$G$-stable} if $\ind_G(\Sigma)=0$.
\end{definition}

In the case when $\Sigma\subset M$ is a $G$-equivariant, \emph{two-sided}, free boundary minimal surface, 
any $G$-equivariant section $Y\in\Gamma_G(N\Sigma)$ of the normal bundle can be written as $Y=u\nu$, 
where $u\in C^\infty(\Sigma)$ and $\nu$ is a choice of unit normal on $\Sigma$.
Then, the second variation of the area of $\Sigma$ along $Y$ is given by
\begin{align*}
Q^\Sigma(Y,Y)=Q_\Sigma(u,u) 
&\vcentcolon=\int_\Sigma \left(\abs{\nabla u}^2 - (\Ric_M(\nu,\nu) + \abs{A}^2)u^2\right) \, d\hsd^2 + \int_{\partial\Sigma} \II^{\partial M}(\nu,\nu)u^2 \, d\hsd^1 \\
&\hphantom{\vcentcolon}=-\int_\Sigma u \mathcal{L}_\Sigma u \,d \hsd^2 
+\int_{\partial \Sigma}\bigl(u\partial_\eta u + \II^{\partial M}(\nu,\nu) u^2 \bigr) \,d \hsd^{1},
\end{align*}
where $\mathcal{L}_\Sigma\vcentcolon=\Delta_\Sigma+\Ric_M(\nu,\nu) + \abs{A}^2$ is the Jacobi operator associated to $\Sigma$.

Now, let us further assume that $\Sigma$ is connected. 
Note that, if $Y=u\nu$ is $G$-equivariant, for all $h\in G$ the fact that $\gelem_*(u\nu) = u\nu$ implies
\[
u\bigl(\gelem(x)\bigr)\nu\bigl(\gelem(x)\bigr)
=d\gelem_x[u(x)\nu(x)] = u(x) \,d\gelem_x[\nu(x)] 
=\operatorname{sgn}_\Sigma(\gelem) u(x) \nu\bigl(\gelem(x)\bigr),
\]
where $\operatorname{sgn}_\Sigma(\gelem) = 1$ if $\gelem_*\nu = \nu$ and $\operatorname{sgn}_\Sigma(\gelem)=-1$ if $\gelem_*\nu = -\nu$. 
Indeed, $h_*\nu $ is equal to $\nu$ or $-\nu$, because $h$ is an isometry and $h(\Sigma)=\Sigma$ (hence $h_*(N\Sigma)=N\Sigma$), and $\Sigma$ is connected (thus the sign is constant). 
Note that $\operatorname{sgn}_\Sigma(\gelem_1\circ\gelem_2) = \operatorname{sgn}_\Sigma(\gelem_1)\operatorname{sgn}_\Sigma(\gelem_2)$.
Therefore, defining
\[
C^\infty_G(\Sigma) \vcentcolon= \{u\in C^\infty(\Sigma) \st u\circ \gelem = \operatorname{sgn}_\Sigma(\gelem) u\ \ \forall \gelem\in G\}
\]
as in \cite{Franz2022}*{Definition~11.1.3}, we have that $Y=u\nu$ is $G$-equivariant if and only if $u\in C^\infty_G(\Sigma)$.
Hence, the $G$-equivariant index of $\Sigma$ coincides with the maximal dimension of a subspace of $C^\infty_G(\Sigma)$ where $Q_\Sigma$ is negative definite.
Moreover, as stated e.\,g.~in \cite{Franz2022}*{Theorem~11.1.5}, the elliptic problem
\begin{equation}\label{eq:EigenvProblem}
\left\{
\begin{aligned}
-\mathcal{L}_\Sigma\varphi&= \lambda \varphi &&\text{ in $\Sigma$, }\\
\partial_\eta \varphi&=-\II^{\partial M}(\nu,\nu)\varphi &&\text{ on $\partial\Sigma$}
\end{aligned}\right.
\end{equation}
admits a discrete spectrum $\lambda_1\le \lambda_2\le \ldots\le \lambda_k\le \ldots\to+\infty$ with associated $L^2_G(\Sigma)$-orthonormal basis of eigenfunctions $(\varphi_k)_{k\ge 1}\subset C^\infty_G(\Sigma)$ of $L^2_G(\Sigma)$. The $G$-equivariant index equals the number of negative eigenvalues and we have the following variational characterization
\begin{equation}\label{eq:VarCharact}
\lambda_k=
\adjustlimits\inf_{S < C^\infty_G(\Sigma),\;\dim S=k~}\sup_{0\neq u\in S}
\frac{Q_\Sigma(u,u)}{\int_\Sigma u^2}.
\end{equation}

\begin{lemma} \label{lem:StableGStable}
Let $\Sigma\subset M$ be a $G$-equivariant free boundary minimal surface and assume that the unit normal vector field $\nu$ on $\Sigma$ is invariant under the $G$-action in the sense that $\gelem_* \nu = \nu$ for all $\gelem\in G$.
Then $\Sigma$ is $G$-stable if and only if it is stable.
\end{lemma}
\begin{proof}
The proof is similar to \cite{Ketover2016Equivariant}*{Proposition~4.6}.
We recall that $\gelem_* \nu = \nu$ for all $\gelem\in G$ if and only if the function $\operatorname{sgn}_\Sigma\colon G\to \{-1,1\}$, used to defined $C^\infty_G(\Sigma)$, is constant equal to $1$.
Let $\psi_1\in C^\infty(\Sigma)$ be the first eigenfunction of the Jacobi operator $\mathcal{L}_\Sigma$, without equivariance. Up to changing sign, we can assume that $\psi_1>0$.
Let us show that $\psi_1\in C^\infty_G(\Sigma)$.
By the variation characterization of eigenvalues, the function $\psi_1\circ\gelem$ is also a first eigenfunction of the Jacobi operator for all $\gelem\in G$. 
Therefore, by uniqueness of the first eigenfunction, $\psi_1\circ\gelem = c \psi_1$ for some constant $c\in\R$. 
In fact $c>0$ because $\psi_1$ is positive. 
Iterating, we have $\psi_1=\psi_1\circ h^{\abs G} = c^{\abs G}\psi_1$, which proves that $c=1$.
As a consequence, for all $\gelem\in G$, we have that $\psi_1\circ\gelem = \psi_1 = \operatorname{sgn}_\Sigma(\gelem)\psi_1$, since $ \operatorname{sgn}_\Sigma(\gelem)=1$ by assumption. This proves that $\psi_1\in C^\infty_G(\Sigma)$ and therefore $\psi_1$ is also the first equivariant eigenfunction. 
\end{proof}

The previous lemma allows us to drop the assumption that $G$ is orientation-preserving in \cite{Franz2023}*{Theorem~4.3}. 

\begin{proposition} \label{prop:ConvergenceOfSurfaces}
Let $\{\Sigma_k\}_{k\in\N}$ be a sequence of $G$-equivariant free boundary minimal surfaces in $M$, with uniformly bounded area, such that $\ind_G(\Sigma_k)\le n$ for some fixed $n\in\N$. Moreover, assume that $\Sigma_k$ intersects orthogonally the singular locus consisting of points of isotropic type $*nn$ for $n\ge 1$. 
Then a subsequence of $\{\Sigma_k\}_{k}$ converges locally graphically and smoothly (possibly with multiplicity) to a free boundary minimal surface $\tilde\Sigma\subset M\setminus(\mathcal{S}_0\cup\mathcal{S}_1\cup\mathcal{Y})$, where $\mathcal{S}_0$ and $\mathcal{S}_1$ are defined in Proposition~\ref{prop:SingLocusStrata} and where $\mathcal{Y}$ is a finite subset of $M$ with $\abs{\mathcal{Y}}\le n\abs{G}$. 
Furthermore, if there exists a $G$-equivariant free boundary minimal surface $\Sigma\subset M$ such that $\tilde\Sigma = \Sigma\setminus (\mathcal{S}_0\cup\mathcal{S}_1\cup\mathcal{Y})$ (namely if $\tilde\Sigma$ extends smoothly to $M$), then $\ind_G(\Sigma)\le n$.
\end{proposition}

\begin{proof}
Let us denote by $\mathcal{S}'$ the union of $\mathcal{S}_0$ and the points in $\mathcal{S}_1$ of isotropic type $nn$. Then, by Lemma~\ref{lem:StableGStable}, the proof of \cite{Franz2022}*{Theorem~4.3} also applies in our setting when considering $M\setminus\mathcal{S}'$. 
Indeed, around any point in $M\setminus\mathcal{S}'$, the notions of stability and $G$-stability are equivalent because $\Sigma_k$ intersects the singular locus on $M\setminus\mathcal{S}'$ orthogonally, and we can apply standard convergence arguments (see \cite{SchoenSimonYau1975}, \cite{GuangLiZhou2020}*{Theorem~1.1}).
\end{proof}

\subsection{Equivariant Dehn's lemma}
 
Given a smooth, simple closed, $G$-equivariant curve $\gamma$ in $\R^3$ let $\Gamma$ be a solution to the associated Plateau problem, i.\,e.~an area minimizing disc with boundary $\partial\Gamma=\gamma$. 
Then $\Gamma$ is not necessarily $G$-equivariant itself. 
A counterexample with $G=\dih_2$ can be found in \cite[§\,2]{Nitsche1968}.
(This phenomenon of symmetry breaking can even be observed one dimension lower: The union of north and south pole on the round sphere $\Sp^2$ is $\Ogroup(2)$-equivariant but any length minimizing geodesic on $\Sp^2$ connecting the poles clearly is not.)
However, under additional assumptions on the prescribed boundary $\gamma$, the following statement holds. 
The lemma applies in any interior geodesic ball with sufficiently small radius and it is a key result to prove the regularity of equivariant min-max limits in the Simon--Smith setting. 
 
\begin{lemma}\label{lem:EquivDehnInterior}
Given any three-dimensional Riemannian manifold $M$, let $B\subset M$ be a convex topological ball in the interior of $M$ 
and let $G$ be a group of isometries which is isomorphic to $nn$ or $\mathord*nn$ for some $n\in\N$ and acts on $B$. 
Let $\gamma\subset\partial B$ be any simple closed $G$-equivariant curve which is disjoint from the line of isotropic type~$G$ in the case $n\geq2$ respectively not contained in the singular locus of the group action in the case $n=1$. 
Then, any area minimizing disc $D$ bounded by $\gamma$ is $G$-equivariant. 
\end{lemma}

\begin{proof}
By Proposition~\ref{prop:SingLocusStrata} and Table~\ref{table:group}, the singular locus of the group action contains a line $\xi\subset B$ of isotropic type~$G$ if $n\geq2$ and, if $G$ is isomorphic $*nn$, additionally a union $\mathcal{S}_2$ of $n$ totally geodesic surfaces $P_1,\ldots,P_n\subset B$ of isotropic type $\mathord*11$. 

Existence, regularity and embeddedness of an area minimizing disc $D$ with boundary $\gamma\subset\partial B$ follows from \cite{AlmgrenSimon1979,MeeksYau1982} (see also \cite{MeeksYau1981}*{Theorem~2}). 
Since the (sub)group of type $nn$ is orientaton-preserving and acts freely on $\gamma$ by assumption, the proof of \cite{MeeksYau1981}*{Theorem~5} shows that any disc $D$ of least area with boundary $\gamma$ inherits its $nn$-equivariance.  
It remains to prove that $D$ is in fact $\mathord*nn$-equivariant in the case where $G$ is isomorphic to $\mathord*nn$. 
Given any $j\in\{1,\ldots,n\}$ let $G_j\cong\mathord*11$ be the subgroup of $G$ which is generated by the reflection $\refl_j$ across the surface $P_j$. 
Since $G$ is generated by $\refl_1,\ldots,\refl_n$, it suffices to prove that $D$ is $G_j$-equivariant. 

Let us first prove the claim in the case when the disc $D$ intersects $P_j$ transversally. 
Then the set $D\cap P_j$ is a disjoint union of smooth connected curves.  
By assumption, the boundary curve $\gamma$ is not contained in $P_j$. 
(In the case $n\geq2$ this follows from the fact that $\gamma$ is disjoint from $\xi$.)
Since $\gamma$ is simple closed and $G_j$-equivariant, $\gamma\cap P_j=\{p_1,p_2\}$.  
The set $D\cap P_j$ necessarily contains a curve $\sigma$ connecting $p_1$ and $p_2$ because $B$ is convex. 
We claim that $D\cap P_j=\sigma$.
Towards a contradiction, we assume that $(D\cap P_j)\setminus\sigma$ is nonempty. 
Since $(D\cap P_j)\setminus\sigma$ is disjoint from $\gamma$ and since $D$ is a topological disc, 
there exists a simple closed curve $\alpha\subset(D\cap P_j)\setminus\sigma$ 
with the property that $D\setminus\alpha$ consists of two connected components $D_1$ and $D_2$ such that $D_1$ is the connected component bounded by $\gamma$ and $\alpha$, and $D_2$ is contained on one side of $P_j$. 
(Indeed, $D_2$ cannot be contained in $P_j$ by the transversality assumption.)
Consider the union $D'=D_1\cup\refl_j D_2$.  
By construction, $D'$ is a topological disc with boundary $\gamma$ and the same area as $D$. 
In particular, it is still area minimizing, and therefore smooth and embedded by \cite{MeeksYau1982}. 
However, $D'$ is not smooth along $\alpha$, which is a contradiction. 

We have shown that $D\cap P_j$ consists of only one curve (the one connecting the two points $p_1,p_2\in\gamma$). 
Hence, the set $D\setminus P_j$ has exactly two connected components. 
Let $\tilde D$ denote the connected component of $D\setminus P_j$ with least area (or any of the two in case they have equal area).
Consider the topological disc $D''=\tilde D\cup\refl_j\tilde D$.
Then $D''$ is a $G_j$-equivariant, area minimizing disc with boundary $\gamma$. 
By \cite{MeeksYau1982}*{Theorem~6}, $D''$ is smooth and embedded. 
Moreover, it coincides with $D$ on $\tilde D$, therefore $D=D''$ by the unique continuation property of minimal surfaces. 
This proves $G_j$-equivariance of $D$.

Let us finally explain how to conclude in the case when the intersection of $D$ and $P_j$ is not transversal. 
In this case, there exists a minimizing sequence $\{D_k\}_{k\in\N}$ of smooth, properly embedded discs $D_k\subset B$ intersecting $P_j$ transversally with boundary $\gamma$, converging (in the parametrized sense of \cite{Morrey1948}*{Definition~2.8}) to $D$ as $k\to\infty$.
For every $k\in\N$ the set $D_k\cap P_j$ consists of one curve $\sigma_k$ connecting the boundary points $p_1$ and $p_2$ defined above, and possibly a finite union of simple closed curves in the interior of $B$. 
In any case, $D_k\setminus\sigma_k$ has exactly two connected components, both of disc type. 
Let $\tilde D_k$ be the connected component of $D_k\setminus\sigma_k$ with least area (or any of the two if both areas are equal) and consider the surface $D_k''=\tilde D_k\cup \refl_j\tilde D_k$. 
By construction $D_k''$ is an immersed, $G_j$-equivariant disc in $B$ with boundary $\gamma$ and area $\hsd^2(D_k'')\leq\hsd^2(D_k)$. 
In particular, $\{D_k''\}_{k\in\N}$ is still a minimizing sequence. 
Moreover, the intersection $D_k''\cap D_k$ contains $\tilde D_k$ for every $k\in\N$. 
By the compactness theorem \cite{Morrey1948}*{Theorem~2.12} and the regularity results in \cite{Morrey1948,AlmgrenSimon1979,MeeksYau1982}, 
a subsequence of $\{D_k''\}_{k\in\N}$ converges to a smooth, embedded, $G_j$-equivariant disc $D''$ with boundary $\gamma$.
Hence, the area $\hsd^2(\tilde D_k)=\frac{1}{2}\hsd^2(D_k'')$ is bounded from below uniformly in $k$, implying that $D''$ coincides with $D$ in an open subset. 
By the unique continuation property of minimal surfaces, $D''$ coincides with $D$ everywhere in $B$, which concludes the proof.
\end{proof}

\begin{lemma} \label{lem:EquivDehnBoundary}
Given any three-dimensional Riemannian manifold $M$ with mean convex boundary $\partial M$, let $B\subset M$ be a topological ball 
such that  $\partial B\setminus\partial M$ is convex and $B\cap\partial M$ a topological disc.
Let $G$ be a group of isometries isomorphic to $\mathord*11$ acting on $B$ and let 
$\gamma\subset\partial B\setminus\partial M$ be a $G$-equivariant curve which is not contained in the singular locus of the group action and connects two points on $\partial M$. 
Then, any area minimizing disc with prescribed boundary $\gamma$ and partially free boundary on $\partial M$ is $G$-equivariant.
\end{lemma}

\begin{proof}
The group $G$ is generated by a reflection $\refl$ and the singular locus of the group action is a totally geodesic surfaces $P$ of isotropic type $\mathord*11$.  
Existence, regularity and embeddedness of an area minimizing disc $D$ with partially free boundary on $\partial M$ and fixed boundary on $\gamma$ follows from \cite{GruterJost1986}*{p.~380 and Section~5}. Alternatively, one can use \cite{Jost1986I}*{Theorems~6.1 and~6.2} (applying the theorems with $X$ being a closed manifold containing $M$, $K=X\setminus M$ such that $\partial K = \partial M$, and $\Gamma=\gamma\subset\partial B\setminus\partial M$). 
See also \cite{MeeksYau1981}*{Theorem~3\,(2)} and the remarks after the theorem for the result in a slightly different setting.

Similarly to the proof of Lemma~\ref{lem:EquivDehnInterior}, it is sufficient to deal with the case when $D$ intersects $P$ transversally. 
In this case, $D\cap P$ consists of a disjoint union of smooth connected curves.
By assumption, $\gamma$ is not contained in $P$. 
Since $\gamma$ is $G$-equivariant, $\gamma\cap P=\{p\}$ consists of a single point. 
Let $\sigma$ be the connected component of $D\cap P$ containing $p$.
Assume by contradiction that $(D\cap P)\setminus\sigma$ is not empty. 
Then, there exists another curve $\alpha \subset (D\cap P)\setminus\sigma$ such that $D\setminus\alpha$ consists of two connected components $D_1$ and $D_2$ with $D_1$ bounded by $\gamma$ and $\alpha$, and $D_2$ contained in one side of $P$. 
Now, one can reach a contradiction as in the proof of Lemma~\ref{lem:EquivDehnInterior}, with the only difference that here $\alpha$ could be a curve connecting two points in $\partial M$ (but disjoint from $\partial B\setminus\partial M$) and $D_2$ could be a disc with partially free boundary on $\partial M$.  

As a result, $D\cap P$ consists of only one connected component $\sigma$, which is a curve connecting the point  $p\in\gamma\cap P$ to a point in $P\cap\partial M$. 
Hence, $D\setminus P$ has exactly two connected components. 
Let $\tilde D$ denote a connected component with least area and consider $D''=\tilde D\cup\refl\tilde D$ as in the proof of Lemma~\ref{lem:EquivDehnInterior}. 
Then, $D''$ is another area minimizing disc, which is $G$-equivariant and coincides with $D$ on an open set. 
Therefore, $D=D''$ by the unique continuation property of minimal surfaces which concludes the proof.
\end{proof}

\subsection{The min-max theorem} \label{subsec:minmax}

Our proof of Theorem~\ref{thm:ExistenceGenusOneCatenoid} uses an equivariant min-max procedure in the Simon--Smith setting. 
This approach allows us to control the topological type of the limiting (free boundary) minimal surface. 
As explained in the introduction, the theory is based on several contributions \cite{Smith1982,ColdingDeLellis2003,DeLellisPellandini2010,Ketover2019,Ketover2016Equivariant,Ketover2016FBMS,Li2015,Franz2022,Franz2023,FranzSchulz2023}. 
In this section, we extend the existing literature by removing the assumption that the action of the finite group $G$ on the ambient manifold $M$ must be orientation-preserving. 
Let us recall the key definitions of sweepout and saturation (c.f.~\cite{FranzSchulz2023}*{Definitions~1.1 and~1.3}), with some minor modifications suitable for our setting. 

\begin{definition}[$G$-sweepout] \label{defn:G-sweepout}
Given a three-dimensional ambient manifold $M$ and a group $G$ of isometries, we say that $\{\Sigma_t\}_{t\in[0,1]^n}$ is a $n$-parameter $G$-sweepout of $M$ if the following properties are satisfied:
\begin{enumerate}[label={\normalfont(\roman*)}]
\item\label{defn:G-sweepout-i} $\Sigma_t$ is a $G$-equivariant subset of $M$ for all $t\in[0,1]^n$.
\item\label{defn:G-sweepout-ii}  Given $t\in[0,1]^n$ the set $\Sigma_t$ is either empty, or a smooth, properly embedded surface in $M$. 
\item\label{defn:G-sweepout-iii} $\Sigma_t$ depends continuously, in the sense of varifolds, on $t\in[0,1]^n$.
\item\label{defn:G-sweepout-iv} 
$\Sigma_t$ depends smoothly on $t$ restricted to the interior of any $k$-face of $[0,1]^n$ for any $k\in\{1,\ldots,n\}$. 
\end{enumerate}
\end{definition}

\begin{remark}
The $k$-faces of $[0,1]^n$ are the $k$-dimensional subsets of $[0,1]^n$ obtained by fixing $n-k$ coordinates to either $0$ or $1$. 
Condition \ref{defn:G-sweepout-iv} allows the topology of $\Sigma_t$ to change when $t\in[0,1]^n$ approaches the boundary of any face from its interior. 
One could relax the definition further by allowing finite sets of points in $M$ and parameters in $[0,1]^n$ where the smoothness in \ref{defn:G-sweepout-ii} and \ref{defn:G-sweepout-iv} may be violated. 
\end{remark}

\pagebreak[2]

\begin{definition}\label{def:SaturationWidth}
Given a $G$-sweepout $\{\Sigma_t\}_{t\in[0,1]^n}$ of $M$, we define its \emph{$G$-saturation} $\Pi$ as the set of all 
$\{\Phi(t,\Sigma_t)\}_{t\in[0,1]^n}$, where 
$\Phi\colon[0,1]^n\times M\to M$ is a smooth map such that $\Phi(t,\cdot)$ is a diffeomorphism of $M$ which commutes with the $G$-action for all $t\in[0,1]^n$.
The \emph{min-max width} of $\Pi$ is then defined as
\[
W_\Pi\vcentcolon=\adjustlimits\inf_{\{\Lambda_t\}\in \Pi~}\sup_{t\in[0,1]^n}\hsd^2({\Lambda_t}).
\]
If a sequence $\{\{\Lambda_t^j\}_{t\in[0,1]^n}\}_{j\in\N}$ in $\Pi$ is minimizing in the sense that 
$\sup_{t\in[0,1]^n}\hsd^2({\Lambda_t^j})\to W_\Pi$ as $j\to\infty$ and if $\{t_j\}_{j\in\N}$ is a sequence in $[0,1]^n$ such that $\hsd^2({\Lambda_{t_j}^j})\to W_\Pi$ as $j\to\infty$, then we call 
$\{\Lambda_{t_j}^j\}_{j\in\N}$ a \emph{min-max sequence}. 
\end{definition}

\begin{remark}
In Definition~\ref{def:SaturationWidth}, we 
do not require $\Phi(t,\cdot)$ do coincide with the identity when $t\in\partial[0,1]^n$, unlike in other references such as \cite{Franz2023,FranzSchulz2023}. 
We emphasize that all elements of the $G$-saturation $\Pi$ are $G$-sweepouts in the sense of Definition~\ref{defn:G-sweepout}. 
Moreover, $\Phi(t,\Sigma_t)$ has the same topology as $\Sigma_t$ for all $t\in[0,1]^n$.
\end{remark}
 
A version of the equivariant min-max theorem (including topological lower semicontinuity) 
has been stated in \cite{FranzSchulz2023}*{Theorems~1.4, 1.8 and~1.9} for the the case where $G$ is a group of \emph{orientation-preserving} isometries -- an assumption which we are able to drop in the following theorem.

\begin{theorem}[Equivariant min-max]\label{thm:EquivMinMax}
Let $M$ be a three-dimensional Riemannian manifold with strictly mean convex boundary and let $G$ be a finite group of isometries of $M$.
Let $\{\Sigma_t\}_{t\in[0,1]^n}$ be a $G$-sweepout of $M$. 
If the min-max width $W_\Pi$ of its $G$-saturation satisfies  
\[
W_\Pi > 0
\]
then there exists a min-max sequence $\{\Sigma^j\}_{j\in\N}$ of (smooth) $G$-equivariant surfaces converging in the sense of varifolds to 
\[
\Gamma\vcentcolon=\sum_{i=1}^\compgamma m_i\Gamma_i,
\]
where the varifolds $\Gamma_1,\ldots,\Gamma_\compgamma$ are induced by pairwise disjoint, connected, embedded free boundary minimal surfaces in $M$ and where the multiplicities $m_1,\ldots,m_\compgamma$ are positive integers.
Moreover, the support of $\Gamma$ is $G$-equivariant and its $G$-equivariant index is less or equal than $n$. 
The first Betti number $\beta_1$ and the genus complexity $\gsum$ are lower semicontinuous along the min-max sequence, in the sense that
\begin{align} \label{eq:LscTopI}
\beta_1(\Gamma)&\leq\liminf_{j\to\infty}\beta_1(\Sigma^j), &
\gsum(\Gamma)&\leq\liminf_{j\to\infty}\gsum(\Sigma^j). 
\end{align}
Finally, if all surfaces in the min-max sequence $\{\Sigma^j\}_{j\in\N}$ are orientable, then the sum of the genus and boundary complexities is lower semicontinuous along the min-max sequence in the sense that 
\begin{align} \label{eq:LscTopII}
\bsum(\Gamma)+\gsum(\Gamma)&\leq\liminf_{j\to\infty}\bigl(\bsum(\Sigma^j)+\gsum(\Sigma^j)\bigr).
\end{align}
\end{theorem}

\begin{remark}\label{rem:complexities}
The genus complexity $\gsum$ and the boundary complexity $\bsum$ have been defined in \cite[Definition~1.6]{FranzSchulz2023}. 
In the connected, orientable case, 
$\gsum$ simply measures the genus, and $1+\bsum$ the number of boundary components of the surface in question.
\end{remark}

\begin{remark}
The width $W_\Pi$ can possibly be realized along the boundary $\partial[0,1]^n$ in the sense that the surfaces $\Sigma^j$ in a min-max sequence may coincide with surfaces $\Lambda_t$ for $t\in\partial[0,1]^n$ of some sweepouts in the saturation. 
In any case, the convergence of the min-max sequence implies that
\(
W_\Pi=\sum_{i=1}^\compgamma m_i\hsd^2(\Gamma_i).
\)
\end{remark}

\begin{remark}\label{rem:B3}
The Euclidean unit ball $\B^3$ is a simply connected ambient manifold with nonnegative Ricci curvature and strictly convex boundary. 
Therefore, every properly embedded free boundary minimal surface in $\B^3$ is necessarily 
\emph{orientable}, 
\emph{connected} by \cite[Lemma~2.4]{FraserLi2014}, 
and has \emph{area} at least $\pi$ by \cite[Theorem~5.4]{FraserSchoen2011}. 
In particular, if $M=\B^3$ in Theorem~\ref{thm:EquivMinMax}, then $\Gamma=m_1\Gamma_1$ with $\hsd^2(\Gamma_1)\geq\pi$.
\end{remark}

\begin{proof}[Proof of Theorem~\ref{thm:EquivMinMax}]
We take \cite{Franz2022}*{Section~13} as a reference for the proof and we describe here the modifications needed to deal with group actions which are not necessarily orientation-preserving. 
Previously, this assumption has been used to prove the local regularity result for min-max surfaces stated in \cite{Franz2022}*{Theorem~13.4.3}. 
Recalling Proposition~\ref{prop:SingLocusStrata} and the notation therein, we distinguish five types of points:
\begin{itemize}
\item Points in $M$ of trivial isotropic type. In this case, the local regularity theory follows from classical results without equivariance (cf.~\cite{ColdingDeLellis2003,Li2015}).
\item Points in $\mathcal{S}_0$. These are isolated points, and therefore the regularity around them follows from the fact that having the good replacement property in annuli is sufficient to prove regularity everywhere (see \cite{ColdingDeLellis2003}*{Proposition~6.3}). Therefore, isolated points are ``removable singularities''.
\item Points in $\mathcal{S}_1$ of isotropic type $nn$ for $n\ge 2$. The local regularity around these points follows from the orientation-preserving case (cf.~\cite{Ketover2016Equivariant,Ketover2016FBMS}) and it is contained in \cite{Franz2022}*{Theorem~13.4.3}.
\item Points in $\mathcal{S}_1$ of isotropic type $*nn$ for $n\ge 2$. 
This case is not covered by the regularity theory \cite{Franz2022}*{Section~13}.
However, inspection of the arguments in \cite{Ketover2016Equivariant}*{Section~4.3} and \cite{Ketover2016FBMS}*{Section~7.2} reveals that the regularity and genus control may be deduced by inputting Lemmata~\ref{lem:StableGStable}, \ref{lem:EquivDehnInterior} and \ref{lem:EquivDehnBoundary} in place of their counterparts for orientation-preserving isometries.
\item Points in $\mathcal{S}_2$. The regularity around these points can be treated as for points at the free boundary, contained in \cite{Li2015}.  
\end{itemize}

This concludes the proof of the regularity of the resulting surface. The topological lower semicontinuity of $\beta_1$, $\gsum$ and $\bsum+\gsum$ now follows as in \cite{FranzSchulz2023}.
Finally, the proof that the support of $\Gamma$ has $G$-equivariant index less or equal to $n$ is the same as in \cite{Franz2023}. 
Indeed, the only point where we used the orientation-preserving assumption in \cite{Franz2023} is in Theorem~4.3 (see Remark~1.11 therein), which is substituted by Proposition~\ref{prop:ConvergenceOfSurfaces} in this setting.
\end{proof}

\section{Surfaces with prismatic symmetry} \label{sec:prismaticsurfaces}

In this section, we study the general structure of $\pri_2$-equivariant surfaces in $\B^3$. 
The first two lemmata concern their topology. 
We recall that the prismatic group $\pri_2$ of order $8$ is generated by the reflections across the three coordinate planes (cf.~Remark~\ref{rem:prismatic}). 
In particular, the octant 
\begin{align}\label{eqn:octant}
\Omega\vcentcolon=\{(x_1,x_2,x_3)\in\B^3\st x_1,x_2,x_3\geq0\} 
\end{align}
is a fundamental domain for the $\pri_2$-action on $\B^3$. 

\begin{lemma}\label{lem:boundary}
Let $\Sigma\subset\B^3$ be any smooth, connected, properly embedded, $\pri_2$-equivariant (not necessarily minimal) surface with boundary. 
Then $\Sigma$ is either a flat disc or the number of its boundary components is even. 
\end{lemma}

\begin{proof}
The singular locus of the $\pri_2$-action coincides with the the union of the flat discs $D_i=\B^3\cap\{x_i=0\}$ over $i\in\{1,2,3\}$. 
Being smooth, connected and $\pri_2$-equivariant, $\Sigma$ either coincides with $D_i$ or intersects $D_i$ orthogonally. 
In the first case the proof concludes. 
In the second case, we restrict $\Sigma$ to the octant $\Omega$ defined in \eqref{eqn:octant} obtaining a properly embedded surface $\tilde{\Sigma}\vcentcolon=\Sigma\cap\Omega$ in $\Omega$ with piecewise smooth boundary. 
Let $\gamma$ be one of the connected components of $\partial\B^3\cap\tilde{\Sigma}$ and let $\tilde\gamma\subset\partial\Sigma$ be the orbit of $\gamma$ under the $\pri_2$-action. 
Then $\gamma$ is either a smooth, closed curve, or a smooth curve connecting two points on the edges of the spherical triangle $\partial\B^3\cap\Omega$, excluding the vertices.
In particular, $\gamma$ is disjoint from at least one of the three edges of $\partial\B^3\cap\Omega$. 
We conclude as follows:  
\begin{itemize}[nosep]
\item If $\gamma$ is closed, then $\tilde\gamma$ has $8$ connected components. 
\item If $\gamma$ connects two points on the same edge, then $\tilde\gamma$ has $4$ connected components. 
\item If $\gamma$ connects two points on different edges, then $\tilde\gamma$ has $2$ connected components. \qedhere
\end{itemize}
\end{proof}

The ambient coordinate functions restricted to any free boundary minimal surface $\Sigma$ in the unit ball are harmonic and satisfy the Robin boundary condition $\partial_\eta u=u$ on $\partial\Sigma$. 
Hence, they necessarily are Steklov eigenfunctions with eigenvalue one. 
A conjecture by Fraser and Li \cite{FraserLi2014} states that $\sigma=1$ is actually the \emph{first} Steklov eigenvalue on any compact, embedded free boundary minimal surface in the unit ball. 
The following result confirms this conjecture for $\pri_2$-equivariant genus one catenoids in $\B^3$. 

\begin{proposition} \label{prop:Steklov}
Let $\Sigma\subset\B^3$ be any embedded, $\pri_2$-equivariant free boundary minimal surface with genus one and two boundary components. 
Then its first Steklov eigenvalue is equal to $1$. 
\end{proposition}

\begin{proof}
As defined in \eqref{eqn:octant}, the fundamental domain $\Omega$ of the $\pri_2$-action on $\B^3$ is a four-sided wedge bounded by three planes and an octant of $\partial\B^3$. 
We claim that $\tilde\Sigma=\Sigma\cap\Omega$ is simply connected with piecewise smooth boundary consisting of exactly five edges which intersect $\partial\Sigma$ in a single connected curve. 
We may then conclude by applying \cite{McGrath2018}*{Theorem~4.2}.

Since $\Sigma$ is necessarily connected (see Remark~\ref{rem:B3}), a similar argument as for \cite[Lemma~7.2]{GirouardLagace2021} proves that $\tilde\Sigma$ is also connected. 
Moreover, $\tilde\Sigma$ meets every face of $\partial\Omega$ orthogonally because $\Sigma$ is smooth, $\pri_2$-equivariant and satisfies the free boundary condition. 
In particular, the exterior angles at the $j$ vertices of $\partial\tilde\Sigma$ are given by the angles between the faces of $\partial\Omega$, which all coincide with $\frac{\pi}{2}$. 
Let $K$ denote the Gauss curvature of $\Sigma$ and $\kappa$ the geodesic curvature along the smooth pieces of $\partial\tilde\Sigma$. 
The Gauss--Bonnet theorem implies that the Euler characteristics $\chi(\Sigma)$ and $\chi(\tilde\Sigma)$ are related by the equation 
\begin{align*}
-4\pi=
2\pi\,\chi(\Sigma)
&=\int_{\Sigma}K+\int_{\partial\Sigma}\kappa
=8\biggr(\int_{\tilde\Sigma}K+\int_{\partial\tilde\Sigma}\kappa\biggr)
=8\Bigl(2\pi\,\chi(\tilde\Sigma)-j\tfrac{\pi}{2}\Bigr)
\end{align*}
which implies $4\chi(\tilde\Sigma)=j-1$. 
On the one hand, $\chi(\tilde\Sigma)\leq1$ because $\tilde\Sigma$ is connected with nonempty boundary. 
On the other hand, $0\leq j\neq1$ because every smooth piece of $\partial\tilde\Sigma$ is either a closed curve or a curve connecting \emph{two} vertices. 
Hence, $\chi(\tilde\Sigma)=1$ and $j=5$ as claimed. 
Finally, since $\Sigma$ has two boundary components, the same argument as in the proof of Lemma~\ref{lem:boundary} reveals that $\tilde\Sigma\cap\partial\B^3$ consists of a single connected curve. 
Therefore, \cite{McGrath2018}*{Theorem~4.2} applies which concludes the proof. 
\end{proof}

The next proposition concerns the $\pri_2$-equivariant index of a $\pri_2$-equivariant free boundary minimal surface with genus one and two boundary components.
Before proving the result, we determine a ``null-direction'' for the Jacobi operator $\mathcal{L}_\Sigma=\Delta_\Sigma+\abs{A_\Sigma}^2$ on free boundary minimal surfaces $\Sigma\subset\B^3$. 
A similar computation (for the function $v^\perp$ given a fixed vector $v$ in place of $x^\perp$) can be found in the proof of \cite{FraserSchoen2016}*{Theorem~3.1}. 

\begin{lemma}\label{lem:xperp}
Let $\Sigma$ be any free boundary minimal surface in $\B^3$ and $\nu$ a choice of unit normal on $\Sigma$. 
Then the function $x^\perp\vcentcolon=\sk{x,\nu}$ satisfies $\mathcal{L}_\Sigma(x^\perp) = 0$ in $\Sigma$ and $x^\perp=0$ on $\partial\Sigma$.
\end{lemma}

\begin{proof}
Denote by $\nabla$ and $D$ the covariant derivatives on $\Sigma$ and $\B^3$, respectively.
For any vector field $X$ tangent to $\Sigma$, we have that
\begin{equation*}
X(x^\perp) = X \sk {x,\nu} = \sk{D_X x,\nu}+\sk{x,D_X \nu} = \sk{X,\nu}-A(X, \hat x)= -A(X, \hat x) = 
\sk{S(\hat x),X} ,
\end{equation*}
where $\hat x = x - x^\perp \nu$ is the tangential component of $x$ and $S(\hat x) = D_{\hat x}\nu$.
Hence, \(\nabla x^\perp = S(\hat x)\).
To compute the Laplacian of $x^\perp$ we consider a normal frame $\{E_1,E_2\}$ on $\Sigma$ at a point $p$. 
Then, at $p$, we have
\begin{align}
\Delta_\Sigma x^\perp &=\operatorname{div}_\Sigma\bigl(S(\hat x)\bigr)
=\sum_{i=1}^2\sk[\big]{\nabla_{E_i}(S(\hat x)),E_i}
=\sum_{i=1}^2\sk[\big]{D_{E_i}(D_{\hat x}\nu),E_i} 
\notag\\
&=\sum_{i=1}^2 \sk[\big]{\nabla_{\hat x}(S(E_i)),E_i} + \sum_{i=1}^2 \sk{S(E_i),\nabla_{E_i}\hat x - \nabla_{\hat x}E_i} 
\label{eqn:star}\\
&=\hat x(H) + \sum_{i=1}^2 \sk[\big]{S(E_i),\nabla_{E_i}\hat x - 2\nabla_{\hat x}E_i}
\notag\\
&=\sum_{i=1}^2 \sk[\big]{S(E_i),\nabla_{E_i}{\hat x}} 
 =\sum_{i=1}^2 \sk[\big]{S(E_i),D_{E_i}{x}-D_{E_i}{(x^\perp \nu)}}
\notag\\
&= H-\sum_{i=1}^2\sk[\big]{S(E_i),D_{E_i}(x^\perp \nu)}= -\sum_{i=1}^2 
\sk[\big]{S(E_i),S(E_i)}x^\perp = -\abs A^2 x^\perp.
\notag
\end{align}
For \eqref{eqn:star} we used that $D_{E_i}D_X\nu = D_XD_{E_i} \nu + D_{[E_i,X]}\nu$, since $\B^3$ is flat. 
We also used that $H=0$ and $\nabla_{\hat x} E_i = (D_{\hat x} E_i)^\top=0$.
This proves that $\mathcal{L}_\Sigma (x^\perp) = 0$ in $\Sigma$. 
The fact that $x^\perp=0$ on $\partial\Sigma$ follows from the free boundary condition.
\end{proof}

\begin{proposition} \label{prop:EquivIndexMoreThan2}
Any $\pri_2$-equivariant, embedded, free boundary minimal surface $\Sigma\subset\B^3$ with genus one and two boundary components has $\pri_2$-equivariant index at least $2$.
\end{proposition}

\begin{proof}
Given $\Omega$ as in \eqref{eqn:octant} we recall from the proof of Proposition \ref{prop:Steklov} that $\tilde\Sigma=\Sigma\cap\Omega$ is simply connected with piecewise smooth boundary consisting of exactly five edges which intersect $\partial\Omega\cap\partial\B^3$ in a single connected curve. 
In particular, one of three planar faces of $\partial\Omega$ contains two (nonconsecutive) edges of $\partial\tilde\Sigma$ which in turn implies that one of the coordinate axes 
contains two (consecutive) vertices of $\partial\tilde\Sigma$. 
Without loss of generality we may assume this to be the  $x_1$-axis. 

The $\pri_2$-equivariant index of $\Sigma$ coincides with the index of $\tilde\Sigma$ subject to the Neumann boundary condition on $\partial\tilde\Sigma\setminus\partial\B^3$, and the usual Robin boundary condition as in \eqref{eq:EigenvProblem} on $\partial\tilde\Sigma\cap\partial\B^3$.

Let $\nu$ be a choice of unit normal vector on $\Sigma$ and consider the function $u(x)=x^\perp\vcentcolon=\sk{x,\nu}$. 
Then $\mathcal{L}_{\Sigma}u=0$ in $\Sigma$ and $u=0$ on $\partial\Sigma$ by Lemma~\ref{lem:xperp}. 
Moreover, $u\in C^\infty_{\pri_2}(\Sigma)$ and thus $\tilde{u}\vcentcolon=u|_{\tilde\Sigma}$ satisfies the Neumann boundary condition along $\partial\tilde\Sigma\setminus\partial\B^3$.
Therefore, $\tilde{u}$ is an eigenfunction with eigenvalue zero for the Jacobi operator $\mathcal{L}_{\tilde\Sigma}$ on $\tilde\Sigma$ subject to the Dirichlet boundary condition on $\partial\tilde\Sigma\cap\partial\B^3$ and the Neumann boundary condition on $\partial \tilde\Sigma\setminus\partial\B^3$.
Moreover, $\tilde{u}$ has at least two nodal domains because as shown above, $\tilde\Sigma$ intersects the $x_1$-axis twice with opposite normal vector. 
By Courant's nodal domain theorem (see also \cite{Franz2022}*{Theorem~11.2.1}), $\tilde{u}$ is not the first eigenfunction of $\mathcal{L}_{\tilde\Sigma}$ and therefore the second eigenvalue of $\mathcal{L}_{\tilde\Sigma}$ with these boundary conditions is less or equal than $0$. 

Let $\varphi_1,\varphi_2 \in C^\infty_{\pri_2}(\Sigma)$ be the first two $\pri_2$-equivariant eigenfunctions of $\mathcal{L}_{\Sigma}$ with Dirichlet boundary conditions on $\partial\Sigma$, and let $\lambda_1^D\le\lambda_2^D$ be the associated eigenvalues.
Observe that $\varphi_1|_{\tilde\Sigma}$ and $\varphi_2|_{\tilde\Sigma}$ coincide with the first two eigenfunctions of $\mathcal{L}_{\tilde\Sigma}$ with Neumann boundary condition on $\partial\tilde\Sigma\setminus\partial \B^3$ and with Dirichlet boundary condition on $\partial\tilde\Sigma\cap\partial\B^3$. 
By the previous argument, we have that $\lambda_1^D\le\lambda_2^D\le 0$.
Moreover, any linear combination $u=a\varphi_1+b\varphi_2$ with $(a,b)\in\R^2\setminus\{0\}$ satisfies 
\begin{align*}
Q_\Sigma(u,u) &= -\int_\Sigma u \mathcal{L}_\Sigma u \,d \hsd^2 
+ \int_{\partial \Sigma}\bigl(u\partial_\eta u + \II^{\partial M}(\nu,\nu) u^2 \bigr) \,d \hsd^{1} 
= -\int_\Sigma u \mathcal{L}_\Sigma u \,d \hsd^2 \\
&= \int_\Sigma (a\varphi_1+b\varphi_2) (\lambda_1^Da\varphi_1+\lambda_2^Db\varphi_2) \,d \hsd^2 \\
&= \int_\Sigma \bigl(\lambda_1^D(a\varphi_1)^2+\lambda_2^D(b\varphi_2)^2\bigr) \,d \hsd^2
\le \lambda_2^D\int_\Sigma u^2\, d\hsd^2,
\end{align*}
where we used that $\varphi_1,\varphi_2$ are $L^2(\Sigma)$-orthogonal.

Now we argue that the second $\pri_2$-equivariant eigenvalue $\lambda_2$ of $\mathcal{L}_{\Sigma}$ subject to the Robin boundary condition on $\partial\Sigma$ (the standard ones as in \eqref{eq:EigenvProblem}, given by the operator associated to the second variation of the area) is strictly less than $0$. 
This is classical and follows from the variation characterization of the eigenvalues, cf.~\eqref{eq:VarCharact}. 
Indeed, 
\[
\lambda_2=
\adjustlimits\inf_{S<C^\infty_{\pri_2}(\Sigma),\;\dim S=2~}\sup_{0\not = u\in S}
\frac{Q_\Sigma(u,u)}{\int_\Sigma u^2}
\leq\sup_{0\not = u = a\varphi_1 + b\varphi_2}\frac{Q_\Sigma(u,u)}{\int_\Sigma u^2} 
=\lambda_2^D \leq 0.
\]
We conclude $\lambda_2<0$ because otherwise $\varphi_2$ would also be an eigenfunction subject to the Robin boundary condition, but this would imply $\varphi_2=0$. 
This proves that the $\pri_2$-equivariant index of $\Sigma$ is at least $2$, as claimed.
\end{proof}

\begin{proposition}\label{prop:IsolatedCritCat}
The critical catenoid is isolated in the space of embedded, $\pri_2$-equivariant free boundary minimal surfaces in $\B^3$, with respect to the varifold topology. 
\end{proposition}

\begin{proof}
Assume by contradiction that there is a sequence $\{\Lambda^k\}_{k\in\N}$ of embedded, $\pri_2$-equivariant free boundary minimal surfaces, different than the critical catenoid, converging in the sense of varifolds to the critical catenoid $\K$ (with multiplicity one). 
Then, by Allard's regularity theorem (see \cite{Allard1972} and \cite{GruterJost1986Allard} for the result in the interior and at the boundary, respectively), 
the convergence of $\Lambda^k$ to the critical catenoid is smooth.
Therefore, $\Lambda^k$ is an embedded, $\pri_2$-equivariant free boundary minimal annulus for all sufficiently large $k$, which implies that $\Lambda^k$ coincides with the critical catenoid by \cite{McGrath2018}*{Theorem~1}.

In fact one can also argue as follows, without using the uniqueness theorem for the critical catenoid. 
By \cite{ColdingMinicozzi2000}*{Lemma~A.1}, the smooth convergence $\Lambda^k\to\K$ as $k\to\infty$ implies the existence of a $\pri_2$-equivariant Jacobi field on the critical catenoid. 
However, the critical catenoid does not have any $\pri_2$-equivariant Jacobi field by \cite{MaximoNunesSmith2017}*{Proposition~6.15}. 
This proves that the sequence $\{\Lambda^k\}_{k}$ is eventually constant.
\end{proof}

\section{Construction of two-parameter sweepouts}
\label{sec:sweepout}

The critical catenoid $\K\subset\B^3$ is the intersection of $\B^3$ with the complete catenoid in $\R^3$ which has been rescaled such that it intersects $\partial\B^3$ orthogonally.  
The following lemma is inspired by \cite[Appendix~A]{CarlottoFranzSchulz2022} and contains the construction of an optimal one-parameter $\Ogroup(2)$-sweepout for the critical catenoid.    

\begin{lemma}\label{lem:catenoid}
Given any $a>0$ and $h\in\interval{0,1}$ the surface $K_{a,h}\subset\R^3$ of revolution obtained by rotating the graph of the function $r_a\colon[-h,h]\to\R$ given by 
\begin{align*}
r_a(z)&=\frac{\cosh(a z)}{\cosh(a h)}\sqrt{1-h^2}  
\end{align*}
around the $x_3$-axis has the following properties: 
\begin{enumerate}[label={\normalfont(\roman*)}]
\item\label{lem:catenoid-i}
$K_{a,h}$ is a smooth, $\Ogroup(2)$-equivariant, properly embedded annulus in $\B^3$ with 
boundary $\partial K_{a,h}=\partial\B^3\cap \{\abs{x_3}=h\}$ and area $\hsd^2(K_{a,h})<2\pi$. 
\item\label{lem:catenoid-ii}
There exist $h_*\in\interval{0,1}$ and $a_*>1$ such that $K_{a_*,h_*}$ coincides with the critical catenoid $\K$. 
\item\label{lem:catenoid-iii} 
There exists a smooth function $\alpha\colon[0,1]\to\interval{1,\infty}$ satisfying $\alpha(h_*)=a_*$ such that $\{K_{\alpha(h),h}\}_{h\in\interval{0,1}}$ is an optimal sweepout for the critical catenoid $\K$ in the sense that 
$\hsd^2(K_{\alpha(h),h})\leq\hsd^2(\K)$ for all $h\in\interval{0,1}$ and 
$\hsd^2(K_{\alpha(h),h})\to0$ for $h\searrow0$ and for $h\nearrow1$. 
\item\label{lem:catenoid-iv}
$K_{a,h}$ converges in the sense of varifolds to $\B^3\cap\{\abs{x_3}=h\}$ as $a\to\infty$.  
\end{enumerate}
\end{lemma}

\begin{proof}
\begin{enumerate}[wide]
\item[\ref{lem:catenoid-i}]
For any $a>0$ and any $h\in\interval{0,1}$ the surface $K_{a,h}$ is $\Ogroup(2)$-equivariant and properly embedded in $\B^3$ because $r_a(h)=r_a(-h)=\sqrt{1-h^2}$ and $r_a(z)=r_a(-z)<\sqrt{1-z^2}$ for every $z\in\interval{-h,h}$.  
It remains to estimate the area of $K_{a,h}$. 
We fix $h\in\interval{0,1}$ and consider the one-parameter family $\{K_{a,h}\}_{a>0}$. 
The corresponding variation vector field vanishes on $\partial K_{a,h}$ and
for every $z\in\interval{-h,h}$ we have 
\begin{align}\notag
\frac{\partial}{\partial a}r_a(z)
&=\frac{z\sinh(az)\cosh(ah)-h\cosh(az)\sinh(ah)}{\cosh^2(a h)}\sqrt{1-h^2}
\\\label{eqn:variation}
&=\bigl(z\tanh(az)-h\tanh(ah)\bigr)r_a(z)<0.
\end{align}
Being a surface of revolution, the mean curvature of $K_{a,h}$ with respect to its outward unit normal vector is given by 
(cf.~\cite[(16)--(17)]{CarlottoFranzSchulz2022})
\begin{align}\label{eqn:mean_curvature}
H&=\frac{r_a r_a''-(r_a')^2-1}{\bigl((r_a')^2+1\bigr)^{\frac{3}{2}}r_a}
=\frac{a^2(1-h^2)\operatorname{sech}^2(ah)-1}{\bigl((r_a')^2+1\bigr)^{\frac{3}{2}}r_a}.
\end{align}
In particular, the sign of $H$ coincides with the sign of $1-F(a,h)$, where 
\begin{align}\label{eqn:F(a,h)}
F(a,h)\vcentcolon=a^{-2}\cosh^2(ah)+h^2. 
\end{align}
Hence, $K_{a,h}$ is a minimal surface if and only if $F(a,h)=1$ and recalling \eqref{eqn:variation}, the first variation formula implies 
\begin{align}\label{eqn:20240421}
\tfrac{\partial}{\partial a}\hsd^2(K_{a,h})
&<0  ~\Leftrightarrow~ F(a,h)>1,
&
\tfrac{\partial}{\partial a}\hsd^2(K_{a,h})
&>0  ~\Leftrightarrow~ F(a,h)<1. 
\end{align}
It is straightforward to check that $a\mapsto F(a,h)$ is strictly convex with $F(a,h)\to\infty$ for $a\searrow0$ and for $a\to\infty$. 
Hence, the equation $F(a,h)=1$ has at most two solutions for $a$. 
Moreover, if $h\in\interval{0,1}$ is sufficiently small, the minimum of $a\mapsto F(a,h)$ is less than $1$ and the equation $F(a,h)=1$ has exactly two solutions for $a$ which we denote by $a_1(h)$ and $a_2(h)$ in increasing order. 
In this case, \eqref{eqn:20240421} implies that $a\mapsto\hsd^2(K_{a,h})$ decreases for $0<a<a_1(h)$ and for $a>a_2(h)$ while it increases for $a_1(h)<a<a_2(h)$. 
For $a\searrow0$ the surface $K_{a,h}$ converges to a round cylinder with area 
$4\pi h\sqrt{1-h^2}\leq2\pi$.  
Consequently, $\hsd^2(K_{a,h})<2\pi$ for all $0<a<a_1(h)$ if the solution $a_1(h)$ exists, and for all $a>0$ otherwise. 
It remains to prove $\hsd^2(K_{a,h})<2\pi$ for $a=a_2(h)$ to conclude. 
Being a surface of revolution, the area of $K_{a,h}$ is given by 
\begin{align}\label{eqn:20240604}
\hsd^2(K_{a,h})&=2\pi\int^{h}_{-h}r_a\sqrt{(r_a')^2+1}\,dz. 
\end{align} 
If $F(a,h)=1$ then $r_{a}(z)=\frac{1}{a}\cosh(a z)$ and 
$r_a(z)\sqrt{(r_a'(z))^2+1}=\frac{1}{a}\cosh^2(a z)$.
A primitive for $z\mapsto\frac{1}{a}\cosh^2(a z)$ is given by $z\mapsto\frac{1}{2}a^{-2}\bigl(az +\sinh(az)\cosh(az)\bigr)$. 
Hence, 
\begin{align}\notag
F(a,h)=1 ~\Rightarrow~
\hsd^2(K_{a,h})&=\frac{2\pi}{a^2} \Bigl(ah+\sinh(ah)\cosh(ah)\Bigr)
\\\label{eqn:f(ah)}
&=2\pi\frac{ah+\sinh(ah)\cosh(ah)}{\cosh^2(ah)+(ah)^2}=\vcentcolon f(ah). 
\end{align}  
The derivative of the map $s\mapsto f(s)$ defined in \eqref{eqn:f(ah)} simplifies to 
\begin{align}\label{eqn:f'(s)}
f'(s)&=4\pi\biggl(\frac{\cosh(s)-s\sinh(s)}{\cosh^2(s)+s^2}\biggr)^2\geq0. 
\end{align}
Hence, $f$ is nondecreasing and for any $s\geq 1$ it is easy to verify $f(s)<2\pi$. 
This proves \ref{lem:catenoid-i}. 

\item[\ref{lem:catenoid-ii}]
Fixing $h\in\interval{0,1}$ it is a priori unclear whether the equation $F(a,h)=1$ has two, one or no solution for $a$. 
Instead we fix $a>1$ and recall from \eqref{eqn:F(a,h)} that the map $h\mapsto F(a,h)$ is increasing with $F(a,0)<1<F(a,1)$. 
Hence, the equation $F(a,h)=1$ has exactly one solution $h=\varphi(a)\in\interval{0,1}$ for any given $a>1$. 
We compute $\frac{\partial F}{\partial h}=2a^{-1}\sinh(ah)\cosh(ah)+2h$ 
and $\frac{\partial F}{\partial a}=2a^{-3}\cosh(ah)\bigl(ah\sinh(ah)-\cosh(ah)\bigr)$.
Since $\frac{\partial F}{\partial h}\neq0$ for $h>0$, the implicit function theorem implies that 
$\varphi\colon\interval{1,\infty}\to\interval{0,1}$ is differentiable and that $a_*$ is a critical point of $\varphi$ with value $h_*=\varphi(a_*)$ if 
$\frac{\partial F}{\partial a}(a_*,h_*)=0$, or equivalently if 
\begin{align}\label{eqn:cat2}
a_*h_*\tanh(a_*h_*)=1.
\end{align}
Equation \eqref{eqn:cat2} determines $\abs{a_*h_*}$ uniquely 
because $x\mapsto x\tanh(x)$ is increasing from $0$ to $\infty$ for $x\geq 0$. 
The equation $F(a_*,h_*)=1$ then implies that $a_*^2=\cosh^2(a_*h_*)+(a_*h_*)^
2$, thus $a_*$ and $h_*$ are also unique.\footnote{Solving the equations numerically yields $a_*h_*
\approx1.19968$ respectively $a_*\approx2.17162$ and $h_*\approx0.55243$.}
Remarkably, the equation $ah\tanh(ah)=1$ is equivalent to the condition that the vectors $\bigl(r_{a}(h),h\bigr)$ and $\bigl(r_{a}'(h),1\bigr)$ are parallel. 
This means that the minimal annulus $K_{a_*,h_*}$ meets $\partial\B^3$ orthogonally and therefore coincides with the critical catenoid $\K$. 

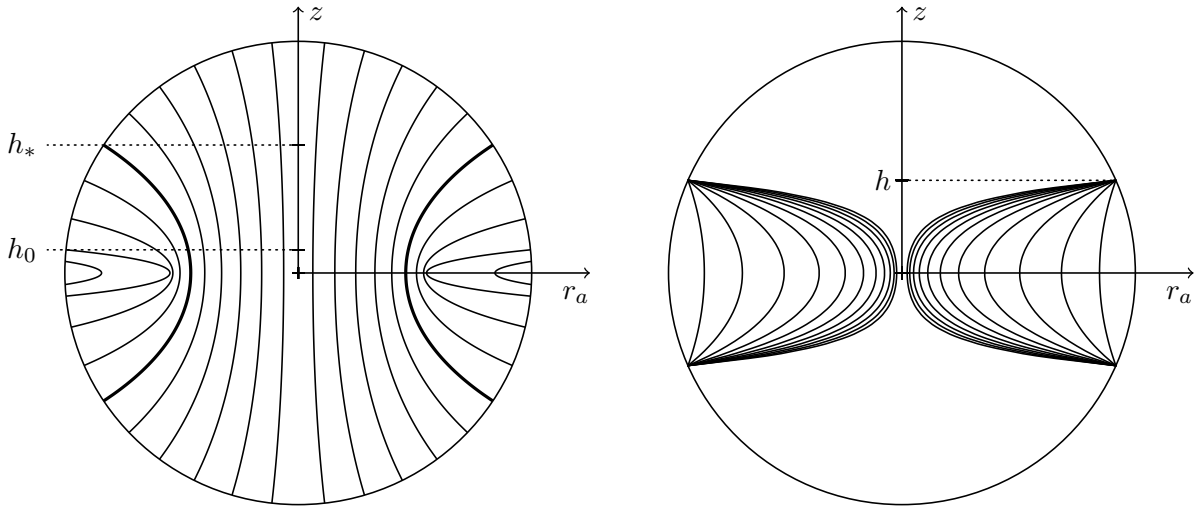
\begin{figure}\centering
\providecommand{\Kah}[3][]{
\draw[domain={-#3}:{#3},smooth,variable=\z,samples=40,#1]
plot({ sqrt(1-#3*#3)*cosh(#2*\z)/cosh(#2*#3)},{\z});
\draw[domain={-#3}:{#3},smooth,variable=\z,samples=40,#1]
plot({-sqrt(1-#3*#3)*cosh(#2*\z)/cosh(#2*#3)},{\z});
}
\pgfmathsetmacro{\xmax}{1.25}
\pgfmathsetmacro{\ymax}{1.15}
\pgfmathsetmacro{\astar}{2.1716}
\pgfmathsetmacro{\hstar}{0.5524}
\pgfmathsetmacro{\hzero}{0.1}
\begin{tikzpicture}[scale=\unitscale,semithick,line join=round,baseline={(0,0)}]
\pgfmathsetmacro{\wdelta}{10}
\pgfmathsetmacro{\wA}{asin(\hstar)-\wdelta}
\pgfmathsetmacro{\wB}{asin(\hstar)-2*\wdelta}
\pgfmathsetmacro{\wC}{asin(\hzero)}
\foreach \winkel in {\wA,\wB,...,\wC}{
\pgfmathsetmacro{\hpar}{sin(\winkel)}
\Kah{\astar*\hstar/\hpar}{\hpar}
}
\pgfmathsetmacro{\wA}{asin(\hstar)+\wdelta}
\pgfmathsetmacro{\wB}{asin(\hstar)+2*\wdelta}
\foreach \winkel in {\wA,\wB,...,90}{
\pgfmathsetmacro{\hpar}{sin(\winkel)}
\Kah{\astar*\hstar/\hpar}{\hpar}
}
\Kah{\astar*\hstar/\hzero}{\hzero}
\Kah{\astar*\hstar/\hzero}{\hzero*1/2}
\Kah[very thick]{\astar}{\hstar}
\draw(0,0)circle(1);
\draw[->](0,0)--(\xmax,0)node[pos=-1,inner sep=0](for_centering){}node[below left={1ex and 0},inner sep=0]{$r_a$};
\draw[->](0,0)--(0,\ymax)node[below right,inner sep=0]{~$z$};
\draw plot[plus](0,0);
\draw[dotted,line cap=round] (0,\hstar)--++(-1.075,0)node[left]{$h_*$};
\draw plot[hdash](0,\hstar);
\draw[dotted,line cap=round] (0,\hzero)--++(-1.075,0)node[left]{$h_0$};
\draw plot[hdash](0,\hzero);
\end{tikzpicture}
\hfill
\pgfmathsetmacro{\hpar}{0.4}
\begin{tikzpicture}[scale=\unitscale,semithick,line cap=round,line join=round,baseline={(0,0)}]
\foreach \apar in {1.01,...,12} 
{
\draw[domain=-\hpar:\hpar,smooth,variable=\z,samples=20]
plot({ sqrt(1-\hpar*\hpar)*cosh(\apar*\z)/cosh(\apar*\hpar)},{\z});
\draw[domain=-\hpar:\hpar,smooth,variable=\z,samples=20]
plot({-sqrt(1-\hpar*\hpar)*cosh(\apar*\z)/cosh(\apar*\hpar)},{\z});
}
\draw(0,0)circle(1);
\draw[->](0,0)--(\xmax,0)node[pos=-1,inner sep=0](for_centering){}
node[below left={1ex and 0},inner sep=0]{$r_a$};
\draw[->](0,0)--(0,\ymax)node[below right,inner sep=0]{~$z$};
\draw plot[plus](0,0);
\draw[dotted] (0,\hpar)--++({sqrt(1-\hpar*\hpar)},0);
\draw plot[hdash](0,\hpar)node[left]{$h$};
\end{tikzpicture}
\caption{Left image: Cross section through the optimal sweepout $\{K_{\alpha(h),h}\}_{h\in\interval{0,1}}$ for the critical catenoid constructed in Lemma~\ref{lem:catenoid}~\ref{lem:catenoid-iii}. \\
Right image: Cross section through the annuli $K_{a,h}$ for fixed $h$ and increasing $a>1$.  
}%
\label{fig:optimal_sweepout}. 
\end{figure}

\item[\ref{lem:catenoid-iii}]
Given $h\in\intervaL{0,1}$ let $\alpha(h)=a_*h_*/h$. 
As explained in the proof of~\ref{lem:catenoid-ii}, the (not necessarily minimal) annulus $K_{\alpha(h),h}$ meets $\partial\B^3$ orthogonally for every $h\in\interval{0,1}$ because $\alpha(h)h\tanh(\alpha(h)h)=a_*h_*\tanh(a_*h_*)=1$. 
We recall from \eqref{eqn:mean_curvature} and \eqref{eqn:F(a,h)} that the sign of the mean curvature of $K_{\alpha(h),h}$ coincides with the sign of $1-F(\alpha(h),h)=1-(h/h_*)^2$. 
Since for every $z\in[-h,h]$ 
\begin{align*}
h\mapsto r_{\alpha(h)}(z)
&=\frac{\cosh\bigl(a_*h_*z/h\bigr)}{\cosh(a_*h_*)}\sqrt{1-h^2} 
\end{align*}
is decreasing in $h\in\interval{0,1}$, the first variation formula implies that $h\mapsto\hsd^2(K_{\alpha(h),h})$ is strictly increasing for $h\in\interval{0,h_*}$ and strictly decreasing for $h\in\interval{h_*,1}$. 
In particular, 
\begin{align}\label{eqn:20240523}
\hsd^2(K_{\alpha(h),h})\leq\hsd^2(K_{a_*,h_*})=\hsd^2(\K)
\end{align}
for every $h\in\interval{0,1}$ with equality if and only if $h=h_*$. 
To prove claim \ref{lem:catenoid-iii}, we choose some $h_0\in\interval{0,h_*}$, replace $\alpha(h)$ with the constant value $a_*h_*/h_0$ for each $h\in\Interval{0,h_0}$ and regularise the resulting Lipschitz function slightly near $h=h_0$ to obtain a smooth function $\alpha\colon[0,1]\to\interval{1,\infty}$. 
Since $\alpha$ is bounded, it is evident from \eqref{eqn:20240604} that $\hsd^2(K_{\alpha(h),h})\to0$ as $h\searrow0$ and if $h_0>0$ is sufficiently small, the area estimate \eqref{eqn:20240523} is preserved. 
For $h\nearrow1$ the profile function $r_{\alpha(h)}$ converges to zero uniformly and $\hsd^2(K_{\alpha(h),h})\to0$ follows again from \eqref{eqn:20240604}. 

\item[\ref{lem:catenoid-iv}]
Let $h\in\interval{0,1}$ be fixed. 
Then $K_{a,h}$ converges in the sense of varifolds to $\B^3\cap\{\abs{x_3}=h\}$ as $a\to\infty$ because 
$\displaystyle\lim_{a\to\infty}r_a(z)=0$ for every $z\in\interval{-h,h}$ 
(see Figure \ref{fig:optimal_sweepout}, right image). \qedhere
\end{enumerate}
\end{proof}

\begin{proposition}\label{prop:sweepout}
There exists a two-parameter $\pri_2$-sweepout $\{\Sigma_{s,t}\}_{(s,t)\in[0,1]^2}$ of $\B^3$ with the following properties:
\begin{enumerate}[label={\normalfont(\roman*)}]
\item\label{prop:sweepout-topo} $\Sigma_{s,t}$ has genus one and two boundary components for all $(s,t)\in\interval{0,1}^2$.
\item\label{prop:sweepout-area} $\hsd^2(\Sigma_{s,0})=\hsd^2(\Sigma_{s,1})=0$ for all $s\in[0,1]$ and $\hsd^2(\Sigma_{s,t}) < 2\pi$ for all $(s,t)\in[0,1]^2$.
\item\label{prop:sweepout-topo0} $\Sigma_{0,t}$ has genus zero and two boundary components, intersects the $x_1$-axis and the $x_2$-axis orthogonally, and is disjoint from the $x_3$-axis for all $t\in\interval{0,1}$.
\item\label{prop:sweepout-topo1} $\Sigma_{1,t}$ has genus zero and two boundary components, intersects the $x_1$-axis and the $x_3$-axis orthogonally, and is disjoint from the $x_2$-axis for all $t\in\interval{0,1}$.
\end{enumerate}
\end{proposition}

\begin{figure}%
\providecommand{\vdist}{\vspace*{1.75ex}}
\pgfmathsetmacro{\sweepoutscale}{\picscale}%
\providecommand{\backsphere}{\path[tdplot_screen_coords,inner color=white,outer color=cyan!5!black!20](0,0)circle(1);\tdplotdrawarc[black!30,thin]{(0,0,0)}{1}{0}{360}{}{}}%
\providecommand{\frontsphere}{\tdplotdrawarc[black,thin]{(0,0,0)}{1}{\phiO-180}{\phiO}{}{}\draw[tdplot_screen_coords,black,thin] (0,0) circle (1);
\pgfresetboundingbox
\useasboundingbox[tdplot_screen_coords](0,0)circle(1);
}
\pgfmathsetmacro{\thetaO}{72}
\pgfmathsetmacro{\phiO}{180-45/2} 
\tdplotsetmaincoords{\thetaO}{\phiO}
~\quad\begin{tikzpicture}[tdplot_main_coords,line cap=round,line join=round,scale={\sweepoutscale*\textwidth/2cm},baseline={(0,0,0)}]
\backsphere  
\FBMS[\sweepoutscale]{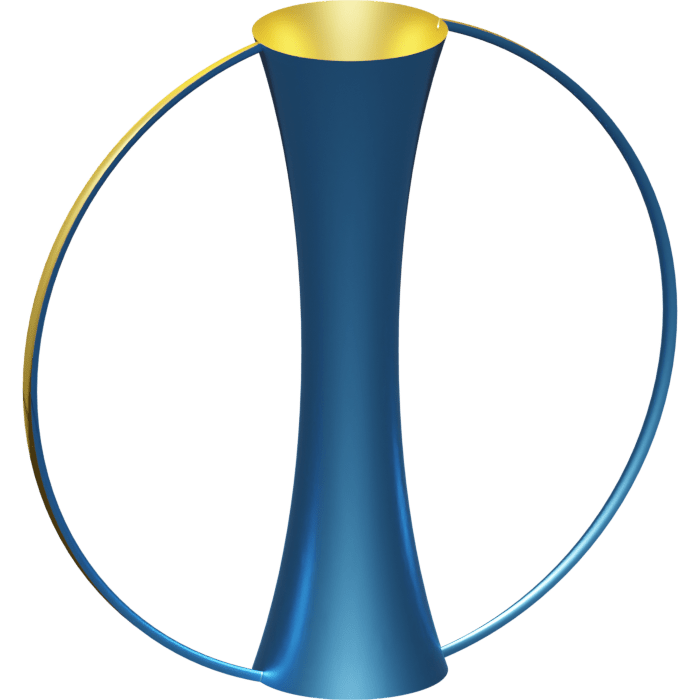}
\frontsphere
\end{tikzpicture}%
\hfill
\begin{tikzpicture}[tdplot_main_coords,line cap=round,line join=round,scale={\sweepoutscale*\textwidth/2cm},baseline={(0,0,0)}]
\backsphere  
\FBMS[\sweepoutscale]{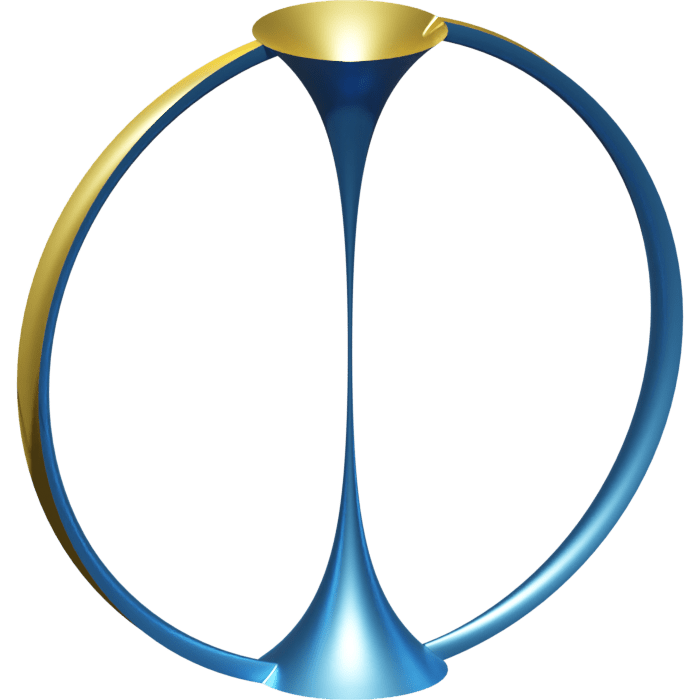} 
\frontsphere
\end{tikzpicture}\quad~%

\vdist

~\quad\begin{tikzpicture}[tdplot_main_coords,line cap=round,line join=round,scale={\sweepoutscale*\textwidth/2cm},baseline={(0,0,0)}]
\backsphere  
\FBMS[\sweepoutscale]{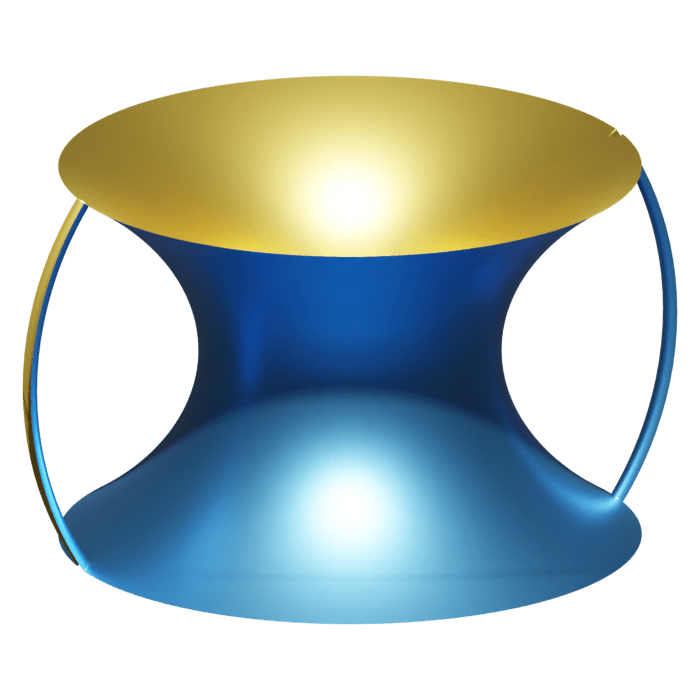}
\draw[densely dotted]
(0.43,0,0)--(0.965,0,0)
(0,0.43,0)--(0,1,0)
(0,0,0.285)--(0,0,1)
;
\draw[->](0.99,0,0)--(1.2,0,0)node[below]{$x_1$};
\draw[->](0,1,0)--(0,1.2,0)node[below]{$x_2$};
\draw[->](0,0,1)--(0,0,1.01)node[below right,inner sep=0]{~$x_3$};
\frontsphere
\end{tikzpicture}%
\hfill
\begin{tikzpicture}[tdplot_main_coords,line cap=round,line join=round,scale={\sweepoutscale*\textwidth/2cm},baseline={(0,0,0)}]
\backsphere  
\FBMS[\sweepoutscale]{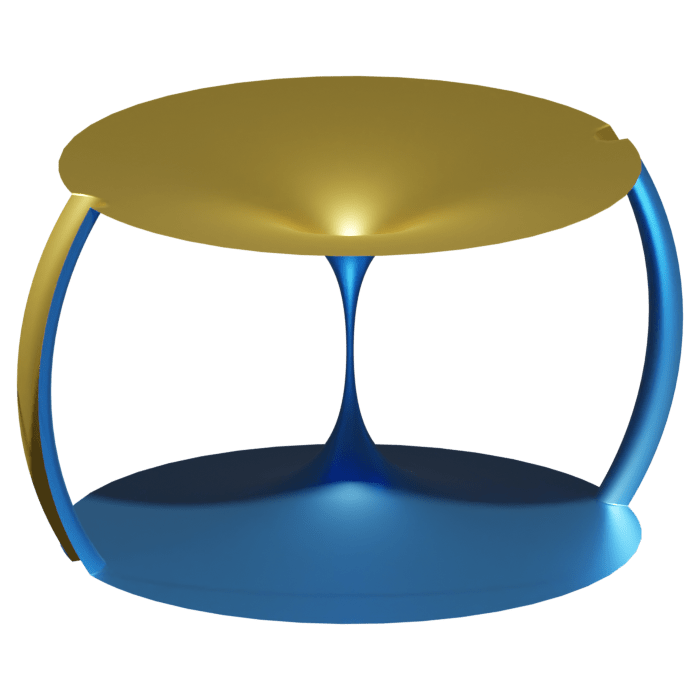}
\frontsphere
\end{tikzpicture}\quad~%

\vdist

~\quad\begin{tikzpicture}[tdplot_main_coords,line cap=round,line join=round,scale={\sweepoutscale*\textwidth/2cm},baseline={(0,0,0)}]
\backsphere  
\FBMS[\sweepoutscale]{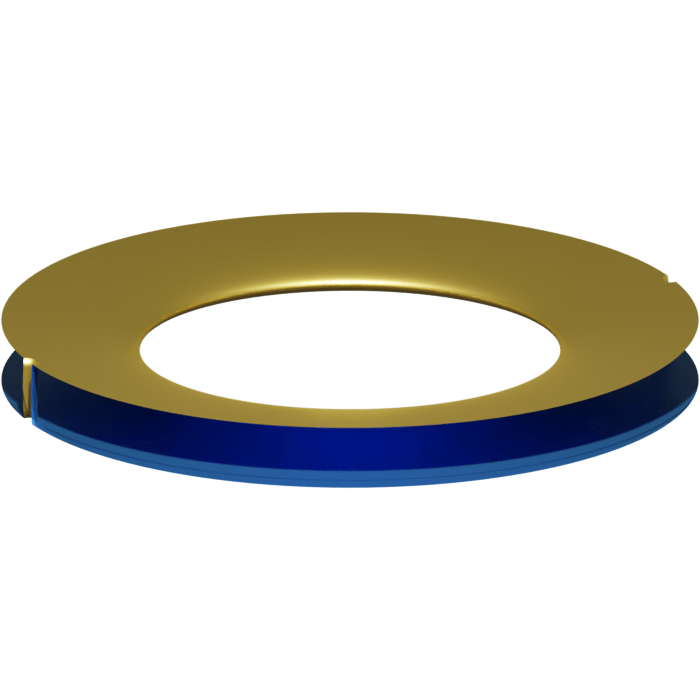}
\frontsphere
\end{tikzpicture}%
\hfill 
\begin{tikzpicture}[tdplot_main_coords,line cap=round,line join=round,scale={\sweepoutscale*\textwidth/2cm},baseline={(0,0,0)}]
\backsphere  
\FBMS[\sweepoutscale]{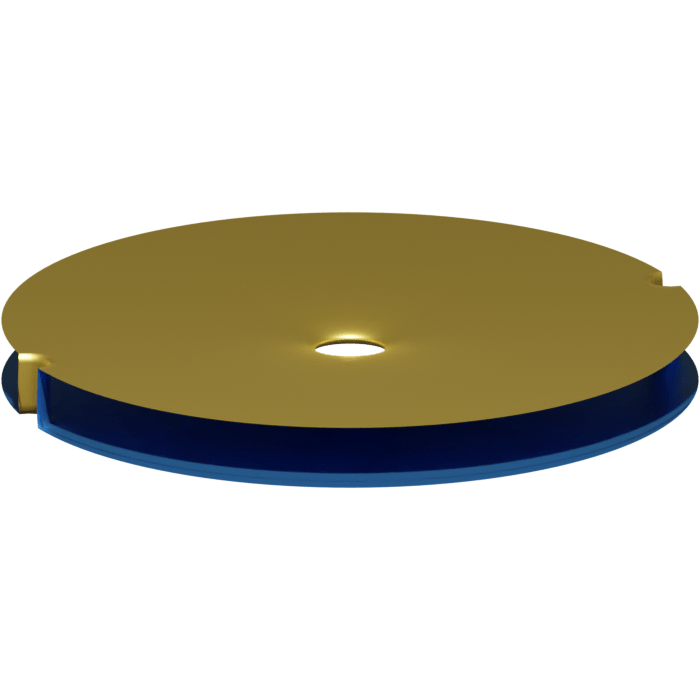}
\frontsphere
\end{tikzpicture}\quad~%
\caption{The surfaces $\Sigma_{s,t}$ defined in \eqref{eqn:Sigma_s,t} for $s$ close to $0$ (left column) respectively $s$ close to $1$ (right column) and three choices of $t\in\Interval{t_0,1}$. 
($s$ increases horizontally, $t$ vertically.) \\
For $s\to0$ the lateral ribbons disappear and an annulus around the $x_3$-axis remains.
For $s\to1$ the central neck disappears and an annulus around the $x_2$-axis remains.
} 
\label{fig:sweepout} 
\end{figure}

\begin{proof}
Given any $a>0$ and $h\in\interval{0,1}$, let $K_{a,h}\subset\R^3$ be as in Lemma~\ref{lem:catenoid} and let 
$\alpha\colon[0,1]\to\interval{1,\infty}$ be as in statement \ref{lem:catenoid-iii} therein. 
For $t\in\interval{0,1}$ we set $\Sigma_{0,t}\vcentcolon=K_{\alpha(t),t}$. 
Then $\Sigma_{0,t}$ has genus zero and two boundary components, intersects the $x_1$-axis and the $x_2$-axis orthogonally, and is disjoint from the $x_3$-axis. 
Let $\beta\colon\Interval{0,1}\times\interval{0,1}\to\interval{1,\infty}$ be a smooth function such that $\beta(0,t)=\alpha(t)$ for every $t\in\interval{0,1}$ and such that $s\mapsto\beta(s,t)$ is increasing for every $t\in\interval{0,1}$ and diverges to $\infty$ as $s\to1$.  
By Lemma \ref{lem:catenoid} 
the area of the surface $K_{\beta(s,t),t}$ satisfies 
\begin{itemize}[nosep]
\item $\hsd^2(K_{\beta(s,t),t})<2\pi$ for every $s,t\in\interval{0,1}$, 
\item $\hsd^2(K_{\beta(s,t),t})\to2\pi(1-t^2)$ as $s\nearrow1$ and 
$\hsd^2(K_{\beta(s,t),t})\to0$ as $t\nearrow1$,
\item $\hsd^2(K_{\beta(0,t),t})\leq\hsd^2(\K)<2\pi$. 
\end{itemize}
Hence, given $t_0>0$ to be chosen there exists $\delta(t_0)>0$ such that for every $(s,t)\in\interval{0,1}\times\Interval{t_0,1}$ 
\begin{align}\label{eqn:area(Kst)}
\hsd^2(K_{\beta(s,t),t})<2\pi-\delta(t_0).
\end{align}
The idea is to attach lateral ribbons to the surface $K_{\beta(s,t),t}$ using a similar construction as in \cite[§\,5]{FranzSchulz2023} (for $n=2$ and $t\in[t_0,1]$). 
Let $B_\varepsilon(p_\pm)\subset\R^3$ denote the ball of radius $\varepsilon>0$ around the point 
$p_\pm=(\pm1,0,0)$. 
Given $\tau\in\interval{-1,1}$ we consider the sets 
\begin{align*}
Q_{\varepsilon}&\vcentcolon=\bigl(B_\varepsilon(p_+)\cup B_\varepsilon(p_-)\bigr)\cap\B^3\cap\{x_3=0\}, & 
\Omega_{\varepsilon}&\vcentcolon=\bigcup_{\tau\in\interval{-1,1}}\Bigl(\bigl(\sqrt{1-\tau^2}Q_{\varepsilon}\bigr)+(0,0,\tau)\Bigr).
\end{align*}
Then $\Omega_{\varepsilon}$ is contained in an $\varepsilon$-neighbourhood of the meridian through $p_\pm$. 
With $0<\varepsilon_0\leq t_0\ll1$ to be chosen, let 
$\varepsilon\colon\intervaL{0,1}\times\Interval{t_0,1}\to\intervaL{0,\varepsilon_0}$ be a smooth function such that $\varepsilon(s,t)\to0$ whenever $s\to0$ or $t\to1$. 
Let $\Upsilon_{s,t}$ be the connected component of $\B^3\setminus K_{\beta(s,t),t}$ containing the poles. 
With $0<t_0\ll 1$ to be chosen, we define 
\begin{align}\label{eqn:Sigma_s,t}
\Sigma_{s,t}&\vcentcolon=\overline{\partial\bigl(\Upsilon_{s,t}\cup\Omega_{\varepsilon(s,t)}\bigr)\setminus\partial\B^3}
\end{align}
for every $(s,t)\in\interval{0,1}\times\Interval{t_0,1}$. 
Here, $\partial X$ denotes the topological boundary of a set $X\subset\R^3$ and $\overline{X}$ its closure. 
If $\varepsilon_0>0$ is sufficiently small $\Sigma_{s,t}$ resembles the union of the surface $K_{\beta(s,t),t}$ with two longitudinal ribbons intersecting the $x_1$-axis orthogonally (see Figure~\ref{fig:sweepout}). 
In particular, $\Sigma_{s,t}$ has the desired topology as stated in \ref{prop:sweepout-topo}. 
The total area contribution of the two ribbons is bounded from above by $c\varepsilon_0$, 
where $c>0$ is an explicit numerical constant (see \cite[Lemma~2.4]{CarlottoFranzSchulz2022} for a similar estimate). 
Choosing $\varepsilon_0\leq \delta(t_0)/(2c)$ and recalling \eqref{eqn:area(Kst)} we obtain  
\begin{align}
\hsd^2(\Sigma_{s,t})\leq 2\pi -\tfrac{1}{2}\delta(t_0)
\end{align}
for all $(s,t)\in\interval{0,1}\times\Interval{t_0,1}$. 
To complete the sweepout in the range $t\in\interval{0,t_0}$ we fix 
$s\in\interval{0,1}$ and consider the following dichotomy. 

If $\hsd^2(\Sigma_{s,t_0})\leq\pi$ (e.\,g.~for small $s$ as in Figure~\ref{fig:sweepout}, bottom left image) it is straightforward 
to $\pri_2$-equivariantly deform the surface $\Sigma_{s,t}$ as $t$ decreases from $t_0$ to $0$ such that its area converges to $0$ as $t\to0$ while staying strictly below $2\pi$ and such that the deformation depends smoothly on~$s$.

If however $\hsd^2(\Sigma_{s,t_0})>\pi$ (e.\,g.~for $s$ near $1$ as in Figure~\ref{fig:sweepout}, bottom right image) the surface $\Sigma_{s,t_0}$ is close to the ``double-disc'' $\B^3\cap\{\abs{x_3}=t_0\}$  near the equator by Lemma~\ref{lem:catenoid}. 
In this case the catenoid estimate (cf.~\cite{KetoverMarquesNeves2020}*{Proposition~2.1 and Theorem~2.4}) can be used to increase the (horizontal) diameter of the ribbons without violating the strict $2\pi$ upper area bound as detailed in the proof of \cite[Theorem~5.1]{FranzSchulz2023}, provided that $t_0>0$ is chosen sufficiently small. 
After this procedure, it is again straightforward to proceed as described in the previous paragraph. 
This completes the proof of claim \ref{prop:sweepout-area}.

For claim \ref{prop:sweepout-topo1} we recall Lemma~\ref{lem:catenoid}~\ref{lem:catenoid-iv} which implies that the surface $K_{\beta(s,t),t}$ converges to the union 
$\B^3\cap\{\abs{x_3}=t\}$ of two discs as $s\to1$. 
Since the ``ribbon radius'' $\varepsilon(s,t)$ is defined to be positive for $s=1$ the surface $\Sigma_{s,t}$ defined in \eqref{eqn:Sigma_s,t} converges as $s\to1$ in the sense of varifolds to a limit with the desired properties. 
\end{proof}

\section{Width estimates and topological control}
\label{sec:topology}

In this final section, we first show width estimates related to the sweepout constructed in Proposition~\ref{prop:sweepout} and then prove Theorem~\ref{thm:ExistenceGenusOneCatenoid}.

\begin{lemma} 
\label{lem:width-1}
Given $\{\Sigma_{s,t}\}_{(s,t)\in[0,1]^2}$ as in Proposition~\ref{prop:sweepout} 
and any curve $\gamma\colon[0,1]\to[0,1]^2$ with $\gamma(0)\in[0,1]\times\{0\}$ and $\gamma(1)\in[0,1]\times\{1\}$, let $\Pi_\gamma$ be the $\pri_2$-saturation of the one-parameter sweepout $\{\Sigma_{\gamma(r)}\}_{r\in[0,1]}$ of $\B^3$. 
Then its width satisfies 
\[\pi<W_{\Pi_\gamma}<2\pi.\]
\end{lemma}

\begin{proof}
The upper bound follows directly from Proposition~\ref{prop:sweepout}~\ref{prop:sweepout-area}. 
For the lower bound, recall that 
for each $(s,t)\in[0,1]\times\interval{0,1}$, the set $\Sigma_{s,t}$ is a (nonempty) smooth surface which does not contain the north pole $(0,0,1)$ and divides $\B^3$ into two connected components. 
Let $F_{s,t}^\Sigma$ denote the connected component of $\B^3\setminus\Sigma_{s,t}$ which does not contain the north pole. 
Setting $F_{s,0}^\Sigma=\emptyset$ and $F_{s,1}^\Sigma=\B^3$, we obtain a continuous two-parameter family $\{F_{s,t}^\Sigma\}_{(s,t)\in [0,1]^2}$ of finite perimeter subsets of $\B^3$. 

Let $\{\Lambda_{t}\}_{t\in[0,1]}\in\Pi_\gamma$ be arbitrary. 
By Definition~\ref{def:SaturationWidth} there exists 
a smooth map $\Phi\colon[0,1]\times\B^3\to\B^3$, where $\Phi(t,\cdot)$ is a diffeomorphism commuting with the $\pri_2$-action for all $t\in[0,1]$, such that $\Lambda_{t}=\Phi(t,\Sigma_{\gamma(t)})$. 
Let $F_{t}=\Phi(t,F_{\gamma(t)}^\Sigma)$. 
Then, the function $f\colon[0,1]\to\R$ given by 
$f(t)=\hsd^3(F_{t})$ is continuous satisfying $f(0)=0$
and $f(1)=\hsd^3(\B^3)$. 
Hence, there exists $t_\gamma\in\interval{0,1}$ such that 
$f(t_\gamma)=\frac{1}{2}\hsd^3(\B^3)$. 
Since $\Lambda_{t_\gamma}$ is the relative boundary of $F_{t_\gamma}$ in $\B^3$, the isoperimetric inequality 
(cf. \cite[Satz~1]{BokowskiSperner1979}, \cite[Theorem~5]{Ros2005}) implies $\hsd^2(\Lambda_{t_\gamma})\geq\pi$ and hence $W_{\Pi_\gamma}\geq\pi$. 
The strict inequality $W_{\Pi_\gamma}>\pi$ now follows exactly as in the proof of \cite{FranzSchulz2023}*{Theorem~5.1}, part (II) from the stability of the isoperimetric inequality in $\B^3$. 
\end{proof} 

The second lemma in this section describes the behavior of the one-parameter sweepouts associated to the vertical sides $\{0\}\times [0,1]$ and $\{1\}\times [0,1]$ of the two-parameter sweepout $\{\Sigma_{s,t}\}_{(s,t)\in[0,1]^2}$.
Let us recall that a min-max sequence $\{\Sigma^j\}_{j\in\N}$ is called \emph{$\pri_2$-almost minimizing in annuli} if there exists a $\pri_2$-equivariant function $r\colon \B^3\to\R^+$ such that a subsequence of $\{\Sigma^j\}_{j\in\N}$ is $\pri_2$-almost minimizing in $\pri_2$-equivariant annuli centered in any $x\in\B^3$ of radii at most $r(x)$. 
We refer to Definitions~13.2.1, 13.2.2 and 13.2.3 and Proposition~13.5.3 in \cite{Franz2022} for more details. 
Most importantly, it suffices that a min-max sequence is almost minimizing in annuli to prove that its limit is a smooth free boundary minimal surface with topological control (see e.\,g. \cite{Franz2022}*{Proposition~13.5.3}, \cite{FranzSchulz2023}*{Proposition~4.4, Theorem~4.11}). 
 
\begin{lemma}\label{lem:CritCatVerticalSides}
Given $\{\Sigma_{s,t}\}_{(s,t)\in[0,1]^2}$ as in Proposition~\ref{prop:sweepout}, let $\Pi_0$ and $\Pi_1$ be the $\pri_2$-saturations of the one-parameter sweepouts $\{\Sigma_{0,t}\}_{t\in[0,1]}$ and $\{\Sigma_{1,t}\}_{t\in[0,1]}$, respectively. 
Let $\{\Sigma^j\}_{j\in\N}$ be any min-max sequence for $\Pi_0$ (respectively $\Pi_1$) which is $\pri_2$-almost minimizing in annuli. 
Then a subsequence of $\{\Sigma^j\}_{j\in\N}$ converges in the sense of varifolds to the critical catenoid $\K_{x_3}$ around the $x_3$-axis (respectively $\K_{x_2}$ around the $x_2$-axis). 
In particular, $W_{\Pi_0}=W_{\Pi_1}=\hsd^2(\K)$. 
\end{lemma}

\begin{proof}
We prove the claim for $\Pi_0$ while the proof for $\Pi_1$ is analogous. 
Lemma~\ref{lem:width-1} states that $\pi<W_{\Pi_0}<2\pi$. 
By \cite{Franz2022}*{Proposition~13.5.3} (see also \cite{FranzSchulz2023}*{Proposition~4.4}) a subsequence of $\{\Sigma^j\}_{j\in\N}$ (which we do not rename) converges in the sense of varifolds with multiplicity one to an embedded, $\pri_2$-equivariant free boundary minimal surface $\Gamma$ in $\B^3$ which has area $\hsd^2(\Gamma)=W_{\Pi_0}$.   
Here we used Remark~\ref{rem:B3} to rule out higher multiplicity given the upper bound on $W_{\Pi_0}$.  
Since $\Gamma$ has greater area than the equatorial disc, $\Gamma$ is not a topological disc by \cite{Nitsche1985}. 
The surface $\Sigma^j$ has genus zero and two boundary components for all $j\in\N$, being isotopic to one of the surfaces $\Sigma_{0,t}$ from Proposition~\ref{prop:sweepout}~\ref{prop:sweepout-topo0}. 
Hence, the topological lower semicontinuity results \cite{FranzSchulz2023}*{Theorems~1.8--9} imply that $\Gamma$ (being not a disc) has the topology of an annulus. 
The uniqueness result \cite{McGrath2018}*{Theorem~1} then implies that $\Gamma$ coincides with one of the three $\pri_2$-equivariant critical catenoids. 
In particular, $W_{\Pi_0}=\hsd^2(\K)$. 

It remains to prove that $\Gamma$ is disjoint from the $x_3$-axis. 
Assume by contradiction that $\Gamma$ is the critical catenoid around the $x_1$-axis (the case around the $x_2$-axis being analogous).
Let $U\subset\B^3$ be a sufficiently thin tubular neighborhood around $\Gamma$ such that $U\cap\partial\B^3$ has two connected components $N_1$ and $N_2$  which are both topological annuli. 
By \cite[Theorem~4.11]{FranzSchulz2023} we can apply a topological surgery procedure to all surfaces $\Sigma^j$ in the min-max subsequence with sufficienlty large $j$, resulting in $\pri_2$-equivariant surfaces $\tilde\Sigma^j\subset U$
such that the sequence $\{\tilde\Sigma^j\}_{j}$ still converges to $\Gamma$ in the sense of varifolds. 

Since $\Sigma^j$ is a topological annulus, \cite{FranzSchulz2023}*{Lemmata~3.4 and~3.6} imply that every connected component of $\tilde\Sigma^j$ has either the topology of an annulus, a disc or a sphere and that at most one of the connected components is a topological annulus. 
(Indeed, $\gsum(\tilde\Sigma^j)=\gsum(\Sigma^j)=0$ and $\bsum(\tilde\Sigma^j)\leq\bsum(\Sigma^j)=1$.) 
Since the limit $\Gamma$ is also a topological annulus, \cite[Theorem~4.11]{FranzSchulz2023} then implies that exactly one connected component $\hat\Sigma^j$ of $\tilde\Sigma^j$ has annular topology. 
(Indeed, $1=\bsum(\Gamma)\leq\bsum(\tilde\Sigma^j)\leq1$.)

We recall that $\hat\Sigma^j$ is obtained from $\Sigma^j$ through surgery in the sense of \cite[Definition~3.1~(c)]{FranzSchulz2023}.
Since both surfaces are topological annuli, \cite{FranzSchulz2023}*{equation (4)} implies that all of the performed surgery operations have to disconnect the surface -- otherwise we would lose the annulus and only discs and spheres would remain. 
As a result, all the individual surgery operations have to act separately on the two boundary components $\gamma_1$ and $\gamma_2$ of $\Sigma^j$ and we may not cut away the half-neck between the two boundary components. 
In particular, the surgery procedure (resulting in $\pri_2$-equivariant surfaces) has to preserve the property that $\gamma_1$ is invariant under reflection $\refl$ with respect to the plane $\{x_3=0\}$. 
However, the two boundary components $\hat\gamma_1,\hat\gamma_2\subset N_1\cup N_2$ of $\hat\Sigma^j$ satisfy $\refl\hat\gamma_1=\hat\gamma_2$ because 
$\refl N_1=N_2$ and we obtain a contradiction. 
\end{proof}

\begin{lemma}
\label{lem:width-2}
Let $\{\Sigma_{s,t}\}_{(s,t)\in[0,1]^2}$ be the $\pri_2$-sweepout of $\B^3$ constructed in Proposition~\ref{prop:sweepout} and let 
$W_\Pi$ be the width of its $\pri_2$-saturation $\Pi$. 
Then \[\hsd^2(\K)<W_\Pi<2\pi.\] 
\end{lemma}

\begin{proof}
The upper bound follows again directly from Proposition~\ref{prop:sweepout}~\ref{prop:sweepout-area}. 
Lemma~\ref{lem:CritCatVerticalSides} implies the lower bound $W_\Pi\geq W_{\Pi_0} = W_{\Pi_1}=\hsd^2(\K)$. 
It remains to prove that the inequality is strict.  

Towards a contradiction we assume that $W_\Pi=\hsd^2(\K)$ and follow the methodology from \cite{Ketover2022}*{pp.~16--17}.
Let $\bigl\{\{\Sigma^j_{s,t}\}_{(s,t)\in[0,1]^2}\bigr\}_{j\in\N}$ be a minimizing sequence in $\Pi$. 
By \cite{Ketover2016Equivariant}*{Proposition~3.9} (see also \cite{Franz2022}*{Proposition~5.2}), 
we can assume that the sequence is pull-tight, namely that every min-max sequence $\{\Sigma^j_{s_j,t_j}\}_{j\in\N}$ (i.\,e.~a sequence such that $\hsd^2(\Sigma^j_{s_j,t_j})\to W_\Pi$ as $j\to\infty$) converges (up to a subsequence) in the sense of varifolds to a stationary varifold.

Let $\mathcal{X}$ be the set of embedded, $\pri_2$-equivariant  free boundary minimal surfaces in $\B^3$ with area equal to $\hsd^2(\K)$ and let $\vard$ denote a distance metrizing  varifold convergence (cf.~Definition~\ref{defn:varifolds}).
By Proposition~\ref{prop:IsolatedCritCat}, there exists $\varepsilon>0$ such that the $\vard$-distance between the critical catenoid and any other $\pri_2$-equivariant free boundary minimal surface is at least $3\varepsilon$.  
Given such $\varepsilon>0$ and any $j\in\N$, we consider the following three pairwise disjoint subsets of $[0,1]^2$ (see  Figure~\ref{fig:square}).
\begin{align}\notag
\mathcal{C}_1^j&=\{(s,t)\in\interval{0,1}\times[0,1]\st\vard(\Sigma_{s,t}^j,\mathcal{X}\setminus\{\K_{x_2},\K_{x_3}\})\leq\varepsilon\},
\\\notag
\mathcal{C}_2^j&=\{(s,t)\in\intervaL{0,1}\times[0,1]\st\vard(\Sigma_{s,t}^j,\K_{x_2})\leq\varepsilon\},
\\\label{eqn:C3j}
\mathcal{C}_3^j&=\{(s,t)\in\Interval{0,1}\times[0,1]\st\vard(\Sigma_{s,t}^j,\K_{x_3})\leq\varepsilon\}.
\end{align}
By choice of $\varepsilon$, these sets have positive distance between each other.
Note also the different ranges for $s$ in the definitions. 
In particular, we have that
\begin{itemize} 
\item $\mathcal{C}_1^j,\mathcal{C}_2^j,\mathcal{C}_3^j$ are all disjoint from the horizontal sides $[0,1]\times\{0\}$ and $[0,1]\times\{1\}$;
\item $\mathcal{C}_1^j,\mathcal{C}_2^j$ are disjoint from the vertical side $\{0\}\times[0,1]$;
\item $\mathcal{C}_1^j,\mathcal{C}_3^j$ are disjoint from the vertical side $\{1\}\times[0,1]$. 
\end{itemize}

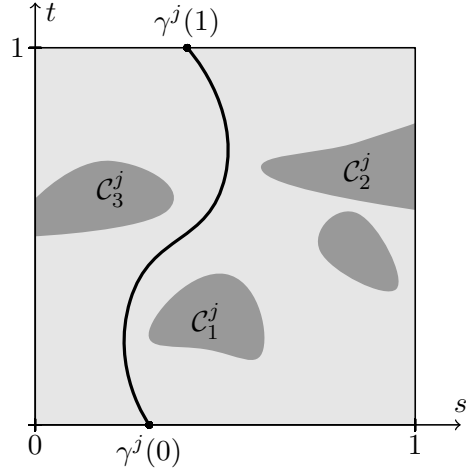
\begin{figure}
\centering 
\begin{tikzpicture}[baseline={(0,0)},scale=5,line cap=round,line join=round,semithick]
\fill[black!10](0,0)rectangle(1,1);
\begin{scope}[black!40]
\fill[smooth]
plot[tension=1] coordinates{
(0,0.6)
(0.2,0.7)
(0.35,0.57)
(0,0.5)
};
\fill[smooth]
plot[tension=1] coordinates{
(1,0.8)
(0.85,0.75)
(0.6,0.67)
(1,0.57)
};
\fill[smooth cycle]
plot[tension=1] coordinates{
(0.95,0.36) 
(0.75,0.46)
(0.85,0.56)
};
\fill[smooth cycle]
plot[tension=1] coordinates{
(0.45,0.2)
(0.6,0.2)
(0.5,0.4)
(0.3,0.25)
};
\end{scope}
\draw(0.45,0.2)node[above]{$\mathcal{C}^j_1$};
\draw(0.2,0.7)node[below]{$\mathcal{C}^j_3$};
\draw(0.85,0.75)node[below]{$\mathcal{C}^j_2$};
\draw plot[bullet](0.3,0)coordinate(g0)node[below]{$\gamma^j(0)$};
\draw plot[bullet](0.4,1)coordinate(g1)node[above]{$\gamma^j(1)$};
\draw[smooth,very thick]
plot[tension=1] coordinates{
(g0)
(0.25,0.33)
(0.5,0.66)
(g1)
};
\draw(0,0)rectangle(1,1);
\draw[->](0,0)--(1.12,0)node[above=1ex,inner sep=0]{$s$~};
\draw[->](0,0)--(0,1.12)node[below right,inner sep=0]{~$t$};
\draw plot[plus] (0,0)node[below]{$0$};
\draw plot[vdash](1,0)node[below]{$1$};
\draw plot[hdash](0,1)node[left]{$1$};
\end{tikzpicture}
\caption{Sets of parameters $(s,t)\in[0,1]^2$ where slices $\Sigma^j_{s,t}$ of the minimizing sequence are close to a surface in $\mathcal{X}$. }%
\label{fig:square}%
\end{figure}

As a result, there exists a smooth curve $\gamma^j\colon[0,1]\to[0,1]^2$
$\setminus(\mathcal{C}_1^j\cup\mathcal{C}_2^j\cup\mathcal{C}_3^j)$
such that $\gamma^j(0)\in[0,1]\times\{0\}$ and $\gamma^j(1) \in[0,1]\times\{1\}$. 
In particular, given $r\in[0,1]$, we have the implication 
\begin{align}\label{eqn:20240903}
\vard(\Sigma_{\gamma^j(r)}^j,\mathcal{X})\leq\varepsilon
~\Rightarrow~\gamma^j(r)\in\{0,1\}\times[0,1].
\end{align}
Let $\Pi_2=\Pi_{\gamma^j}$ be the $\pri_2$-saturation of the one-parameter sweepout $\{\Sigma^j_{\gamma^j(r)}\}_{r\in[0,1]}$ and let $W_{\Pi_2}$ be its width. 
Lemma~\ref{lem:width-1} yields 
$\pi<W_{\Pi_2}$.  
Moreover, $W_{\Pi_2} \leq W_\Pi$ by definition. 
Recalling our assumption $W_\Pi=\hsd^2(\K)$, we distinguish two cases, both leading to contradictions. 
\begin{enumerate}[wide, label={\itshape Case \arabic*:},ref=\arabic*]
\item\label{proof-case1} $W_{\Pi_2}<\hsd^2(\K)$. 
Since $W_{\Pi_2}>0$ we may apply Theorem~\ref{thm:EquivMinMax} to obtain a $\pri_2$-equivariant free boundary minimal surface $\Gamma_2\subset\B^3$ with $\pri_2$-equivariant index at most $1$, satisfying 
$\gsum(\Gamma_2)\leq1$ and 
$\bsum(\Gamma_2)+\gsum(\Gamma_2)\leq2$, 
where we recall Remark~\ref{rem:complexities}. 
This leaves the following possibilities for the topology of $\Gamma_2$. 
\begin{itemize}
\item $(\gsum,\bsum)(\Gamma_2)=(1,1)$. 
Then $\Gamma_2$ has genus one, two boundary components and $\pri_2$-equivariant index equal to $1$ in contradiction with  Proposition~\ref{prop:EquivIndexMoreThan2}.
\item $(\gsum,\bsum)(\Gamma_2)=(1,0)$. 
Then $\Gamma_2$ has genus one with connected boundary, contradicting Lemma~\ref{lem:boundary}.
\item $(\gsum,\bsum)(\Gamma_2)=(0,2)$. 
Then $\Gamma_2$ has three boundary components contradicting Lemma~\ref{lem:boundary}.
\item $(\gsum,\bsum)(\Gamma_2)=(0,0)$. 
Then $\Gamma_2$ is a free boundary minimal disc with area $\hsd^2(\Gamma_2)=W_{\Pi_2}>\pi$ contradicting Nitsche's \cite{Nitsche1985} uniqueness result.  
\item $(\gsum,\bsum)(\Gamma_2)=(0,1)$. 
Then, $\Gamma_2$ is a $\pri_2$-equivariant free boundary minimal annulus with area $\hsd^2(\Gamma_2)<\hsd^2(\K)$ contradicting McGrath's \cite{McGrath2018}*{Theorem~1} uniqueness result.  
\end{itemize}

\item\label{proof-case2} $W_{\Pi_2} = \hsd^2(\K)$. 
Then $\{\{\Sigma^j_{\gamma^j(r)}\}_{r\in[0,1]}\}_{j\in\N}$ is a minimizing sequence in $\Pi_2$. Indeed, recalling that $\bigl\{\{\Sigma^j_{s,t}\}_{(s,t)\in[0,1]^2}\bigr\}_{j\in\N}$ is a minimizing sequence in $\Pi$, we have
\[
\hsd^2(\K)=W_\Pi=\lim_{j\to\infty}\sup_{(s,t)\in[0,1]^2} \hsd^2(\Sigma^j_{s,t})
\geq\lim_{j\to\infty} \sup_{r\in[0,1]} \hsd^2(\Sigma^j_{\gamma^j(r)})
\geq W_{\Pi_2}=\hsd^2(\K)
\]
and therefore all inequalities are equalities in this case.

By \cite{ColdingGabaiKetover2018}*{Lemma~A.1} (see also \cite{Franz2022}*{Lemma~5.6}), $\{\{\Sigma^j_{\gamma^j(r)}\}_{r\in[0,1]}\}_{j\in\N}$ contains a min-max sequence $\{\Sigma^j\}_{j\in\N}$, which is $\pri_2$-almost minimizing in annuli, converging in the sense of varifolds to a free boundary minimal surface $\Gamma_3$ with area equal to $W_{\Pi_2} = \hsd^2(\K)$. 
In particular, $\Gamma_3\in\mathcal{X}$ and hence $\vard(\Sigma^j,\mathcal{X}) \leq\vard(\Sigma^j,\Gamma_3)\le \varepsilon$ 
provided that $j$ is sufficiently large. 
By \eqref{eqn:20240903} we have $\Sigma^j=\Sigma^j_{(0,t_j)}$ or $\Sigma^j=\Sigma^j_{(1,t_j)}$ for some $t_j\in[0,1]$. 
Up to extracting a subsequence we may assume (without loss of generality) that the first case holds for all $j$. 
Then, since $\{\Sigma^j_{(0,t_j)}\}_{j}$ is a min-max sequence that is $\pri_2$-almost minimizing in annuli, Lemma~\ref{lem:CritCatVerticalSides} implies that it converges to $\K_{x_3}$. 
However, recalling \eqref{eqn:C3j}, this contradicts the fact that $(0,t_j)\notin\mathcal{C}_3^j$ for all $j$. 
\qedhere
\end{enumerate}
\end{proof}

\begin{proof}[Proof of Theorem~\ref{thm:ExistenceGenusOneCatenoid}] 
Let $\{\Sigma_{s,t}\}_{(s,t)\in[0,1]^2}$ be the $\pri_2$-sweepout of $\B^3$ from Proposition~\ref{prop:sweepout} and let $\Pi$ be its $\pri_2$-saturation as defined in Definition~\ref{def:SaturationWidth}.
By Lemma~\ref{lem:width-2}, $\hsd^2(\K)<W_\Pi<2\pi$. 
Thus, recalling Remark~\ref{rem:B3}, we may apply Theorem~\ref{thm:EquivMinMax} to obtain the existence of a min-max sequence $\{\Sigma^j\}_{j\in\N}$ converging with multiplicity one to an embedded, $\pri_2$-equivariant free boundary minimal surface $\Gamma$ in $\B^3$ with area 
$\hsd^2(\Gamma)=W_\Pi$ and $\pri_2$-equivariant index at most~$2$.

Every surface in the min-max sequence $\{\Sigma^j\}_{j\in\N}$ has two boundary components and at most genus one, 
i.\,e.~$\bsum(\Sigma^j)=1$ and $\gsum(\Sigma^j)\leq1$, recalling Remark~\ref{rem:complexities}. 
By the topological lower semicontinuity result stated in Theorem~\ref{thm:EquivMinMax}~\eqref{eq:LscTopI}--\eqref{eq:LscTopII}, 
we obtain $\gsum(\Gamma)\leq1$ and $\bsum(\Gamma)+\gsum(\Gamma)\leq2$. 
With the same arguments as in the the proof of Lemma~\ref{lem:width-2}, we can rule out the cases $(\gsum,\bsum)(\Gamma)\in\{(1,0),(0,2),(0,0)\}$. 
Similarly, if $(\gsum,\bsum)(\Gamma)=(0,1)$, then $\Gamma$ is a $\pri_2$-equivariant free boundary minimal annulus with area $\hsd^2(\Gamma)=W_\Pi>\hsd^2(\K)$, contradicting \cite{McGrath2018}*{Theorem~1}.
The only remaining possibility is $(\gsum,\bsum)(\Gamma)=(1,1)$. 
Therefore, $\FKS\vcentcolon=\Gamma$ has genus one and two boundary components as claimed.

Finally, by Proposition~\ref{prop:Steklov}, the first nonzero Steklov eigenvalue of $\FKS$ is equal to $1$ and Proposition~\ref{prop:EquivIndexMoreThan2} implies that the $\pri_2$-equivariant index of $\FKS$ is at least $2$ and hence equal to $2$.  
\end{proof}

 
\bibliography{fbms-bibtex} 
\printaddress

\end{document}